\documentclass[reqno]{amsart}
\usepackage{amsmath}
\usepackage{amssymb}
\usepackage{amsthm}
  \usepackage{paralist}
  \usepackage{graphics} 
  \usepackage{epsfig} 
\usepackage{graphicx}  \usepackage{epstopdf}
\usepackage[colorlinks=true]{hyperref}
\hypersetup{urlcolor=blue, citecolor=red}

\newtheorem{theorem}{Theorem}[section]

\newtheorem{lemma}[theorem]{Lemma}

\theoremstyle{definition}
\newtheorem{definition}[theorem]{Definition}

\title
[A moving boundary problem for swelling]
{Local weak solvability of a moving boundary problem describing swelling along a halfline}



 \email{k.kumazaki@nagasaki-u.ac.jp}
 \email{adrian.muntean@kau.se}

\date{}

\begin{document}
\maketitle

\centerline{\scshape Kota Kumazaki}
\medskip
{\footnotesize
 \centerline{Nagasaki University}
   \centerline{Department of Education}, 
   \centerline{1-14, Bunkyo-cho, Nagasaki, 852-8521, Japan}
}

\medskip

\centerline{\scshape Adrian Muntean}
\medskip
{\footnotesize
 \centerline{Karlstad University}
   \centerline{Department of Mathematics and Computer Science}
   \centerline{Universitetsgatan 2, 651 88 Karlstad, Sweden}
} 

\medskip

%

\begin{abstract}
We obtain the local well-posedness of a moving boundary problem that describes the swelling of a pocket of water  within an infinitely thin elongated pore (i.e. on $[a,+\infty), \ a>0$).  Our result involves fine {\em a priori} estimates of the moving boundary evolution, Banach fixed point arguments as well as an application of the general theory of evolution equations governed by subdifferentials.
\end{abstract}

\section{Introduction}
\label{intro}

We wish to understand which effect the water-triggered micro-swelling of pores can have at observable scales of concrete-based materials. Such topic is especially relevant in cold regions, where buildings exposed to extremely low temperatures undergo freezing and build microscopic ice lenses that ultimately lead to the mechanical damage of the material; see, for instance,  \cite{Setzer}.  One way to tackle this issue from a theoretical point of view is to get a better picture of the transport of moisture.  Our long-term goal is to build a macro-micro model for moisture transport suitable for cementitious mixtures, where at the macroscopic scale the transport of moisture follows a porous-media-like equation, while at the microscopic scale the moisture is involved in an adsorption-desorption process leading to a strong local swelling of the pores.  Such a perspective would lead to a system of partial differential equations with distributed microstructures, see \cite{FMA,FT} for related settings.  In this paper, we propose a one-dimensional microscopic problem posed on a halfline with a moving boundary at one of the ends. The moving boundary conditions encode the swelling mechanism, while a diffusion equation is responsible to providing water content for the swelling to take place.


Since we are interested in how far the water content can actually push the {\em a priori unknown} moving boundary of swelling,  we assume that pore depth is infinite although the actual physical length is finite. Our target here is to show the well-posedness of the model. 

Let us now describe briefly  the setting of our equations. The timespan is $[0,T]$ while the pore is $[a,+\infty)$, with $a, T\in (0,+\infty)$. The variables are $t\in [0,T]$ and $z\in [a,+\infty)$. 
The boundary $z=a$ denotes the edge of the pore in contact with wetness.  The interval $[a, s(t)]$ indicates the region of diffusion of the  water content $u(t)$, where $s(t)$ is the moving interface of the water region.  The function $u(t)$ acts in the non-cylindrical region $Q_s(T)$ defined by
\begin{align*}
& Q_s(T):=\{(t, z) | 0<t<T, \ a<z<s(t) \}. 
\end{align*}

Our free boundary problem, which we denote by $(\mbox{P})_{u_0, s_0, h}$, reads: 

Find the pair $(u(t,z),s(t))$ satisfying
\begin{align*}
& u_t-ku_{zz}=0 \mbox{ for }(t, z)\in Q_s(T), \tag{1.1}\\
& -ku_z(t, a)=\beta(h(t)-Hu(t, a)) \mbox{ for }t\in(0, T), \tag{1.2}\\
& -ku_z(t, s(t))=u(t, s(t))s_t(t) \mbox{ for }t\in (0, T), \tag{1.3}\\
& s_t(t)=a_0(u(t, s(t))-\varphi(s(t))) \mbox{ for }t\in (0, T), \tag{1.4}\\
& s(0)=s_0, u(0, z)=u_0(z) \mbox{ for }z \in [a, s_0]. \tag{1.5}
\end{align*}
Here $k$ is a diffusion constant, $\beta$ is a given adsorption function on $\mathbb{R}$ that is equal to 0 for negative input and takes a positive value for positive input, $h$ is a given moisture threshold function on $[0, T]$,  $H$ and $a_0$ are further given (positive) constants, $\varphi$ is our swelling function defined on $\mathbb{R}$, while  $s_0$ and $u_0$ are the initial data. 

From the physical perspective, (1.1) is the diffusion equation displacing $u$ in the unknown region $[a, s]$; the boundary condition (1.2),  imposed at $z=a$,  implies that the moisture content $h$ inflows if $h$ is present  at $z=a$ in a larger amount than $u$. The boundary condition (1.3) at $z=s(t)$ describes the mass conservation at the moving boundary. Indeed, if the flux $u_z(t, a)$ at $z=a$ is active on the time interval $[t, t+\Delta t]$ for $t>0$, namely, $s_t(t)>0$, then, it holds that 
\begin{align*}
& \int_a^{s(t)} u(t, z) dz -ku_z(t, a) \Delta t =\int_a^{s(t+\Delta t)} u(t+\Delta t, z)dz. 
\end{align*}
Hence, by dividing $\Delta t$ in both side and letting $\Delta t \to 0$ we formally obtain that 
\begin{align*}
& -ku_z(t, a) = \int_a^{s(t)} u_t(t, z) dz + s_t u(t, s(t)). 
\end{align*}
By $u_t=ku_{zz}$ in (1.1), we derive that 
\begin{align*}
-ku_z(t, a) & = \int_a^{s(t)} ku_{zz}(t, z)dz + s_t u(t, s(t)) \\
& = ku_z(t, s(t))-ku_z(t, a)+s_tu(t, s(t)). 
\end{align*}
This formal argument motivates the structure of the moving boundary condition (1.3). The ordinary differential equation (1.4) describes the growth rate of the free boundary $s$ and it is determined by the balance between the water content $u(t, s(t))$ at $z=s(t)$ and the swelling expression $\varphi(s(t))$.  It is worth mentioning at this stage that the function $\varphi(s(t))$ limits the growth of the moving boundary.

From the mathematical point of view, our free boundary problem resembles remotely the classical one phase Stefan problem and its variations for handling superheating, phase transitions, evaporation; compare  \cite{Evaporation, disso-precipi, Helmig, Xie} and references cited therein. Our work contributes to the existing mathematical modeling work of swelling by Fasano and collaborators (see \cite{FMP,FM}, e.g.)  as well as other authors cf. e.g. \cite{Z}. 
 The main difference between these papers and our formulation lies in the choice of the boundary conditions (1.2) and (1.3). Most of the cited settings  impose an homogeneous Dirichlet boundary condition at one of the boundaries, while  we  impose flux boundary conditions at both boundaries. Relation (1.2) will be used in a forthcoming work to connect the microscopic moving boundary discussed here to a macroscopic transport equation. 
 
 It is worth mentioning that the literature contains already a number of free boundary problems posed for the corrosion of  porous materials. We review here the closest contributions to our setting. For instance,  we refer to Muntean and B$\ddot{\mbox{o}}$hm \cite{MB}  who proposed a well-posed free boundary problem as mathematical model for the concrete carbonation process in one space dimension; Aiki and Muntean \cite{AM, AM1, AM2} proved the existence and uniqueness of a solution for a simplified Muntean-B$\ddot{\mbox{o}}$hm-model and obtained the large-time behavior of the free boundary as $t \to \infty$. Also, in \cite{AMSS, SAMS},  Sato et al. proposed a free boundary problem as a mathematical model of single pore adsorption, a setting very close to ours, and showed the existence of a solution locally in time;   Aiki and Murase guaranteed in \cite{AM} the existence of a solution globally in time and established the large time behaviour of this solution. Recently, based on the results of Sato et al. \cite{SAMS} and Aiki and Murase \cite{AI-Mur},  Kumazaki et al.  proposed in \cite{KASM} a multiscale model of moisture transport with adsorption, coupling in a particular fashion a macroscopic diffusion equation with the microscopic picture of the model proposed  by Sato et al.  in  \cite{SAMS} and ensured  the local existence of a solution of this two-scale  problem. We refer the reader to \cite{FMA, FT, MN} and references cited therein for comprehensive descriptions of modeling, mathematical analysis and numerical approximation of  reaction-diffusion systems posed on multiple space scales in the absence of free or moving boundaries.
 

The paper is organized as follows: In Section 2, we state the used notation and assumptions as well as  our main theorem concerning the existence and uniqueness of a solution for the moving boundary problem. In Section 3, we consider an auxiliary problem focused on finding $u$ for given $s$ and prove the existence of a solution of this problem by relying on the abstract theory of evolution equations governed by time-dependent subdifferentials. By using the result of Section 4, we finally prove our main theorem by suitably applying  Banach's fixed point theorem and the maximum principle. 

\section{Notation and assumptions}
\label{na}

In this framework, we use the following basic notations. We denote by $| \cdot |_X$ the norm for a Banach space $X$. The norm and the inner product of a Hilbert space $H$ are
 denoted by $|\cdot |_H$ and $( \cdot, \cdot )_H$, respectively. Particularly, for $\Omega \subset {\mathbb R}$, we use the standard notation of the usual Hilbert spaces $L^2(\Omega)$, $H^1(\Omega)$ and $H^2(\Omega)$.

Throughout this paper, we assume the following restrictions on the model parameters and functions:

(A1) $a$, $a_0$, $H$, $k$ and $T$ are positive constants. 

(A2) $h\in W^{1,2}(0, T)\cap L^{\infty}(0, T)$ with $h\geq 0 $ on $(0, T)$. 

(A3) $\beta\in C^1(\mathbb{R)}\cap W^{1, \infty}(\mathbb{R})$ such that $\beta=0$ on $(-\infty, 0]$ and $\beta$ is bounded and $\beta' \geq 0$ on $\mathbb{R}$. Also, we put $c_{\beta}=\mbox{sup}_{r\in \mathbb{R}}\beta(r) + \mbox{sup}_{r\in \mathbb{R}}\beta'(r)$. 


(A4) $\varphi \in C^1(\mathbb{R})\cap W^{1, \infty}(\mathbb{R})$ such that $\varphi=0$ on $(-\infty, 0]$, $\varphi \geq 0$ on $[0, +\infty)$, $\varphi' \geq 0$ on $\mathbb{R}$ and $\mbox{sup}_{r\in \mathbb{R}}\varphi(r)\leq \mbox{min}\{2\varphi(a), |h|_{L^{\infty}(0, T)}H^{-1}\}$. Also, we put $c_{\varphi}=\mbox{sup}_{r\in \mathbb{R}}\varphi(r) + \mbox{sup}_{r\in \mathbb{R}}\varphi'(r)$. 

(A5) $s_0>a$ and $u_0\in H^1(a, s_0)$ such that $\varphi(a)\leq u_0(z) \leq |h|_{L^{\infty}(0, T)}H^{-1}$ on $[a, s_0]$.

For $T>0$, let $s$ be a function on $[0, T]$ and $u$ be a function on $Q_s(T):=\{(t, z) | 0\leq t\leq T, a<s(t)\}$.

Next, we define our concept of solution to (P)$_{u_0, s_0, h}$ on $[0,T]$ in the following way:

\begin{definition}
\label{def}
 We call that pair $(s, u)$ a solution to (P)$_{u_0, s_0, h}$ on $[0, T]$ if the following conditions (S1)-(S6) hold:

(S1) $s$, $s_t\in L^{\infty}(0, T)$, $a<s$ on $[0, T]$, $u\in L^{\infty}(Q_s(T))$, $u_t$, $u_{zz}\in L^2(Q_s(T))$ and $t \in [0, T] \to |u_z(t,  \cdot)|_{L^2(a, s(t))}$ is bounded;

(S2) $u_t-ku_{zz}=0$ on $Q_s(T)$;

(S3) $-ku_z(t, a)=\beta(h(t)-Hu(t, a))$ for a.e. $t\in [0, T]$;

(S4) $-ku_z(t, s(t))=u(t, s(t))s_t(t)$ for a.e. $t\in [0, T]$;

(S5) $s_t(t)=a_0(u(t, s(t))-\varphi(s(t)))$ for a.e. $t\in [0, T]$;

(S6) $s(0)=s_0$ and $u(0, z)=u_0(z)$ for $z \in [a, s_0]$.
\end{definition}

The main result of this paper is concerned with the existence and uniqueness of a  locally in time solution  in the sense of Definition \ref{def} to the problem (P)$_{u_0, s_0, h}$. This result is stated in the next Theorem. 

\begin{theorem}
\label{t1}
Let $T>0$. If (A1)-(A5) hold, then there exists $T^*<T$ such that (P)$_{u_0, s_0, h}$ has a unique solution $(s, u)$ on $[0, T^*]$ satisfying 
$\varphi(a)\leq u \leq |h|_{L^{\infty}(0, T)}H^{-1}$ on $Q_s(T^*)$.
\end{theorem}

To be able to prove Theorem \ref{t1}, we transform (P)$_{u_0, s_0, h}$, initially posed in a non-cylindrical domain, to  a cylindrical domain. Let $T>0$. For given $s\in W^{1,2}(0, T)$ with $a<s(t)$ on $[0, T]$, we introduce the following new function obtained by the indicated change of variables, "freezing" the moving domain:
\begin{align*}
& \tilde{u}(t, y)=u(t, (1-y)a+ys(t)) \mbox{ for } (t, y)\in Q(T):=(0, T)\times (0, 1). 
\end{align*}
By using the function $\tilde{u}$, we consider now the following problem $(\mbox{P})_{\tilde{u}_0, s_0, h}$:
\begin{align*}
& \tilde{u}_t(t, z)-\frac{k}{(s(t)-a)^2}\tilde{u}_{yy}(t, z)=\frac{ys_t(t)}{s(t)-a}\tilde{u}_y(t, z) \mbox{ for }(t, z)\in Q(T), \tag{2.1}\\
& -\frac{k}{s(t)-a}\tilde{u}_y(t, 0)=\beta(h(t)-H\tilde{u}(t, 0)) \mbox{ for }t\in(0, T), \tag{2.2}\\
& -\frac{k}{s(t)-a}\tilde{u}_y(t, 1)=\tilde{u}(t, 1)s_t(t) \mbox{ for }t\in (0, T),  \tag{2.3}\\
& s_t(t)=a_0(\tilde{u}(t, 1)-\varphi(s(t)) \mbox{ for }t\in (0, T), \tag{2.4}\\
& s(0)=s_0, \tag{2.5}\\
& \tilde{u}(0, y)=\tilde{u}(0, y)=u_0(1-y)a+y s(0))(:=\tilde{u}_0(y)) \mbox{ for }y \in [0, 1]. \tag{2.6}
\end{align*}

\begin{definition}
For $T>0$, let $s$ be functions on $[0, T]$ and $\tilde{u}$ be a function on $Q(T)$, respectively. We call that a pair $(s, \tilde{u})$ is a solution of (P)$_{\tilde{u}_0, s_0, h}$ on $[0, T]$ if the conditions (S'1)-(S'2) hold:

(S'1) $s$, $s_t\in L^{\infty}(0, T)$, $a<s$ on $[0, T]$, $\tilde{u}\in W^{1,2}(Q(T))\cap L^{\infty}(0, T; H^1(0, 1))\cap L^2(0, T;H^2(0, 1))\cap L^{\infty}(Q(T))$.

(S'2) (2.1)--(2.6) hold.
\end{definition}

\begin{theorem}
\label{t2}
Let $T>0$. If (A1)-(A5) hold, then there exists $T^*<T$ such that (P)$_{\tilde{u}_0, s_0, h}$ has a unique solution $(s, \tilde{u})$
on $[0, T^*]$.
\end{theorem}
In the rest of the paper, we focus on proving Theorem \ref{t2}, which finally will turn to provide candidate solutions for Theorem \ref{t1}.

\section{Auxiliary Problem (\mbox{AP})}
\label{AP}

In this section, we first prove Theorem \ref{t2}. Let $T>0$, $L>a$ and $s\in W^{1, 2}(0, T)$ with $a<s<L$ on $[0, T]$,  we introduce the following auxiliary problem $(\mbox{AP})_{\tilde{u}_0, s, h}$: 
\begin{align*}
& \tilde{u}_t(t, z)-\frac{k}{(s(t)-a)^2}\tilde{u}_{yy}(t, z)=\frac{ys_t(t)}{s(t)-a} \tilde{u}_y(t, z) \mbox{ for }(t, z)\in Q(T), \tag{3.1}\\
& -\frac{1}{s(t)-a}\tilde{u}_y(t, 0)=\beta(h(t)-H\tilde{u}(t, 0)) \mbox{ for }t\in(0, T), \tag{3.2}\\
& -\frac{1}{s(t)-a}\tilde{u}_y(t, 1)=a_0\tilde{u}(t, 1)(\tilde{u}(t, 1)-\varphi(s(t))) \mbox{ for }t\in (0, T),  \tag{3.3}\\
& \tilde{u}(0, y)=\tilde{u}_0(y) \mbox{ for }y \in [0, 1]. \tag{3.4}
\end{align*}

In order to study $(\mbox{AP})_{\tilde{u}_0, s, h}$, for given $f\in W^{1,2}(Q(T))\cap L^2(0, T; H^1(0, 1))$, we consider firstly the following problem called $(\mbox{AP})_{\tilde{u}_0, f, s, h}$. This reads:
\begin{align*}
& \tilde{u}_t(t, z)-\frac{k}{(s(t)-a)^2}\tilde{u}_{yy}(t, z)=\frac{ys_t(t)}{s(t)-a} f_y(t, z) \mbox{ for }(t, z)\in Q(T), \\
& -\frac{1}{s(t)-a}\tilde{u}_y(t, 0)=\beta(h(t)-H\tilde{u}(t, 0)) \mbox{ for }t\in(0, T), \\
& -\frac{1}{s(t)-a}\tilde{u}_y(t, 1)=a_0\tilde{u}(t, 1)(\tilde{u}(t, 1)-\varphi(s(t))) \mbox{ for }t\in (0, T), \\
& \tilde{u}(0, y)=\tilde{u}_0(y) \mbox{ for }y \in [0, 1]. 
\end{align*}

Now, we introduce a family $\{ \psi^t \}_{t\in[0, T]}$ of time-dependent functionals $\psi^t: L^2(0, 1) \to \mathbb{R}\cup \{+\infty \}$
for $t\in [0, T]$, defined by 
$$
\psi^t(u):=\begin{cases}
             \displaystyle{\frac{k}{2(s(t)-a)^2} \int_0^1|u_y(y)|^2 dy} + \displaystyle{\frac{1}{s(t)-a}\int_0^{u(1)} a_0 \sigma(\xi)(\sigma(\xi)-\varphi(s(t)) d\xi}\\
             -\displaystyle{\frac{1}{s(t)-a}\int_0^{u(0)}\beta(h(t)-H\xi)d\xi} \mbox{ if } u\in D(\psi^t),\\
             +\infty \mbox{ if othewise},
             \end{cases}
$$
where 
$$
 \sigma(r)=\begin{cases}
               r \hspace{5mm}\mbox{ if } r>\varphi(a), \\
               \varphi(a) \mbox{ if } r\leq \varphi(a), 
               \end{cases}
$$
and $D(\psi^t)=\{z\in H^1(0, 1) | z\geq 0 \mbox{ on } [0, 1] \}$ for $t \in [0, T]$. 
What concerns the function $\psi^t$, we prove the following structural properties. 

\begin{lemma}
\label{lem1}
Let $s\in W^{1,2}(0, T)$ with $a<s(t)<L$ on $[0, T]$. Assuming (A1)-(A5), then the following statements hold: 
\begin{itemize}
\item[(1)] There exists positive constant $C_0$ and $C_1$ such that the following inequalities hold:
\begin{align*}
(i) & \ |u(0)|^2 \leq C_0 \psi^t(u)+C_1 \mbox{ for }u\in D(\psi^t)\\
(ii) & \ |u(1)|^2 \leq C_0 \psi^t(u)+C_1 \mbox{ for }u\in D(\psi^t)\\
(ii) & \ \frac{k}{2(s(t)-a)^2}|u_y|^2_{L^2(0, 1)} \leq C_0 \psi^t(u)+C_1 \mbox{ for }u\in D(\psi^t)
\end{align*}
\item[(2)] For $t\in [0, T]$, the functional $\psi^t$ is proper, lower semi-continuous, and  convex  on $L^2(0, 1)$.
\end{itemize}
\end{lemma}

\begin{proof}
First, we note that for $t\in [0, T]$ if $u\in D(\psi^t)$ then, $u(0)$ and $u(1)$ are positive.
Let $t\in [0, T]$ and $u\in D(\psi^t)$. Then, if $u(1) >\varphi(a)$, then 
\begin{align*}
& \int_0^{u(1)} a_0\sigma(\xi)(\sigma(\xi)-\varphi(s(t))d\xi \\
=&a_0\varphi^2(a)(\varphi(a)-\varphi(s(t))\\
& +a_0\frac{u^3(1)}{3}-a_0 \varphi(s(t))\frac{u^2(1)}{2}-\biggl(a_0\frac{\varphi^3(a)}{3}-a_0 \varphi(s(t))\frac{\varphi^2(a)}{2}\biggr)\\
\geq &\frac{a_0}{3} u^3(1) (1-2\eta^{3/2})-\frac{a_0}{3}\left(\frac{c_{\varphi}}{2\eta}\right)^3
+\biggl(a_0\frac{2\varphi^3(a)}{3}-a_0 \frac{\varphi^2(a)c_{\varphi}}{2}\biggr),
\tag{3.5}
\end{align*}
where $\eta$ is arbtrary positive constant. By taking  $\eta$ suitably  in (3.5) and putting $\delta_s\leq s(t)-a$ for $t\in [0, T]$, we see that there exists $c_0=c_0(\eta)$, $c_1=c_1(\eta$) such that 
\begin{align*}
& \frac{1}{s(t)-a}\int_0^{u(1)} a_0 \sigma(\xi)(\sigma(\xi)-\varphi(s(t))d\xi \geq \frac{c_0}{L-a}u^3(1)-\frac{c_1}{\delta_s}\geq \frac{c_0\varphi(a)}{L-a}u^2(1)-\frac{c_1}{\delta_s}. 
\tag{3.6}
\end{align*}
In the case $u(1)\leq \varphi(a)$, then $\sigma(u(1))=\varphi(a)$ so that we have the similarly inequality (3.6). Also, we have that 
\begin{align*}
& \frac{-1}{s(t)-a}\int_0^{u(0)} \beta(h(t)-H\xi)d\xi \geq \frac{-c_{\beta}}{s(t)-a}u(0)=\frac{-c_{\beta}}{s(t)-a}\left (u(1)-\int_0^1 u_y(y) dy \right)\\
& \geq -\frac{c_0\varphi(a)}{2(L-a)} u^2(1)-\frac{L-a}{2c_0\varphi(a)}\left(\frac{c_{\beta}}{\delta_s}\right)^2-\frac{k}{4(s(t)-a)^2}\int_0^1 |u_y(y)|^2 dy -\frac{c^2_{\beta}}{k}\\
& \geq -\frac{c_0\varphi(a)}{2(L-a)} u^2(1)-\frac{k}{4(s(t)-a)^2}\int_0^1 |u_y(y)|^2 dy-\left(\frac{L-a}{2c_0\varphi(a)}\left(\frac{c_{\beta}}{\delta_s}\right)^2 + \frac{c^2_{\beta}}{k} \right),  
\tag{3.7}
\end{align*}
where $c_{\beta}$ is the same constant as in (A3). By adding (3.6) and (3.7), it yields
\begin{align*}
\psi^t(u) &\geq \frac{k}{4(s(t)-a)^2} \int_0^1|u_y(y)|^2 dy \\
& + \frac{c_0\varphi(a)}{2(L-a)} u^2(1)-\frac{c_1}{\delta_s}-\left(\frac{L-a}{2c_0\varphi(a)}\left(\frac{c_{\beta}}{\delta_s}\right)^2 + \frac{c^2_{\beta}}{k} \right). 
\tag{3.8}
\end{align*}
Also, it holds that 
\begin{align*}
|u(0)|^2 & =\biggl | \int_0^1 u_y(y) dy + u(1) \biggr|^2 \leq 2\left (\int_0^1 |u_y(y)|^2 dy + |u(1)|^2 \right)\\
& \leq 2\left (\frac{2(L-a)^2}{k} \frac{k}{2(s(t)-a)^2}\int_0^1 |u_y(y)|^2 dy + |u(1)|^2 \right). 
\end{align*}
Therefore, by (3.8) and the estimate of $u(0)$ we see that the statement (1) of Lemma \ref{lem1} holds. 

We now prove statement (2).  For $r\in \mathbb{R}$, put 
\begin{align*}
& g_1(s(t), r)=\frac{1}{s(t)-a}\int_0^r a_0\sigma(\xi)(\sigma(\xi)-\varphi(s(t)) d\xi, \\
& g_2(s(t), h(t), r)=-\frac{1}{s(t)-a}\int_0^r \beta(h(t)-H\xi)d\xi. 
\end{align*}
Then, by $a<s(t)$, $\beta' \geq 0$ in (A3) and (A4) we see that 
$r \mapsto a_0\sigma(r)(\sigma(r)-\varphi(s(t))$ and $r\mapsto -\beta(h(t)-Hr)$ are also monotone increasing. 
This means that $\psi ^t$ is convex on $L^2(0, 1)$. 
Also, the lower semi-continuity of $\psi^t$ is enough to prove that the level set of $\psi^t$ is closed in $L^2(0, 1)$. This is easy to prove by using Lemma \ref{lem1} and the Sobolev's embedding $H^1(0, 1) \hookrightarrow C([0, 1])$ in one dimensional case. Thus, we see that for $t\geq 0$, $\psi^t$ is a proper, lower semi-continuous, convex function on $L^2(0, 1)$.
\end{proof}

By Lemma \ref{lem1} we obtain the following existence result concerning the solutions to problem  $(\mbox{AP})_{\tilde{u}_0, f, s, h}$.
\begin{lemma}
\label{lem2}
Let $T>0$ and $L>a$. If (A1)-(A5) hold, then, for given $s\in W^{1,2}(0, T)$ with $a<s<L$ on $[0, T]$ and  $f\in W^{1,2}(Q(T))\cap L^{\infty}(0, T;H^1(0, 1))$, then the problem $(\mbox{AP})_{\tilde{u}_0, s, f, h}$ admits a unique solution $\tilde{u}$ on $[0, T]$ such that $\tilde{u}\in W^{1,2}(Q(T))\cap L^{\infty}(0, T;H^1(0, 1))$. Moreover, the function $t\to \psi^t(u(t))$ is absolutely continuous on $[0, T]$.
\end{lemma}

\begin{proof}
By Lemma \ref{lem1}, for $t\in [0, T]$ $\psi^t$ is a proper lower semi-continuous convex function on $L^2(0, 1)$ and $\partial \psi^t$ is single valued. With this information at hand, we see that $z^*=\partial \psi^t(u)$ if and only if $z^*\in L^2(0, 1)$ and 
\begin{align*}
& z^*=-\frac{k}{(s(t)-a)^2} u_{yy} \mbox{ on } (0, 1),\\
& -\frac{k}{s(t)-a}u_z(0)=\beta(h(t)-Hu(0)),\\
& -\frac{k}{s(t)-a}u_z(1)=a_0\sigma(u(1))(\sigma(u(1))-\varphi(s(t))). 
\end{align*}
Also, there exists a positive constant $C$ such that for each $t_1$, $t_2\in [0, T]$ with $t_1\leq t_2$, and for any $u\in D(\psi^{t_1})$, there exists $\bar{u}\in D(\psi^{t_2})$ such that 
\begin{align*}
&|\bar{u}-u|_{L^2(0, 1)} \leq |s(t_1)-s(t_2)|(1+|\varphi^{t_1}(u)|^{1/2}), \tag{3.9}\\
&|\varphi^{t_2}(\bar{u})-\varphi^{t_1}(u)|\leq C(|s(t_1)-s(t_2)|+|h(t_1)-h(t_2)|)(1+|\varphi^{t_1}(u)|). \tag{3.10}
\end{align*}
Indeed, by taking $\bar{u}:=u$ it is easy to prove that (3.9) and (3.10) holds. 
Now, we consider the following Cauchy problem (CP):
$$ \begin{cases}
\tilde{u}_t+\partial \psi^t(\tilde{u}(t)) =\frac{y s_t(t)}{s(t)-a}f_y(t) \mbox{ in }L^2(0, 1)\\
\tilde{u}(0, y)=\tilde{u}_0(y) \mbox{ for }y\in [0, 1].
\end{cases}
$$
Here, we notice that  since $f\in L^2(0, T;H^1(0, 1))$ and $s \in W^{1,2}(0, T)$ then $\frac{yf_y(t)s_t(t)}{s(t)-a}\in L^2(0, T; L^2(0, 1))$.  
Then, by the general theory of evolution equations governed by time dependent subdifferentials (see \cite{Kenmochi} and references cited therein),  we conclude that 
(CP) has a solution $\tilde{u}$ on $[0, T]$ such that $\tilde{u}\in W^{1,2}(Q(T))$, $\psi^t(\tilde{u}(t)) \in L^{\infty}(0, T)$ and $t \to \psi^t(\tilde{u}(t))$ is absolutely continuous on $[0, T]$. This implies that $\tilde{u}$ is a unique solution of $(\mbox{AP})_{\tilde{u}_0, f, s, h}$ on $[0, T]$.
\end{proof}

\begin{lemma}
\label{lem3}
Let $T>0$, $L>a$ and $s\in W^{1,\infty}(0, T)$ with $a<s<L$ on $[0, T]$. If (A1)-(A5) hold, then, $(\mbox{AP})_{\tilde{u}_0, s, h}$ has a unique solution $\tilde{u}$ on $[0, T]$ such that $\tilde{u}\in W^{1,2}(Q(T))\cap L^{\infty}(0, T;H^1(0, 1))$. 
\end{lemma}

\begin{proof}
By Lemma \ref{lem2}, we can define the solution operator $\Gamma_T(f)=\tilde{u}$, where 
$\tilde{u}$ is a unique solution of $(\mbox{AP})_{\tilde{u}_0, f, s, h}$ for given $f\in W^{1,2}(Q(T))\cap L^{\infty}(0, T;H^1(0, 1))$. 
Now, for $i=1,2$ we put $\Gamma(f_i)=\tilde{u}_i$ and $f=f_1-f_2$ and $\tilde{u}=\tilde{u}_1-\tilde{u}_2$. 
Then, we have 
\begin{align*}
& \frac{1}{2}\frac{d}{dt}|\tilde{u}|^2_{L^2(0, 1)} - \int_0^1 \frac{k}{(s(t)-a)^2}\tilde{u}_{yy} \tilde{u} dy = \int_0^1 \frac{y s_t}{s(t)-a}f_{y} \tilde{u}dy.\tag{3.11}
\end{align*}
Using the structure of the boundary conditions, we obtain 
\begin{align*}
& - \int_0^1 \frac{k}{(s(t)-a)^2}\tilde{u}_{yy} \tilde{u} dy \\
& = -\frac{k}{(s(t)-a)^2}\tilde{u}_y(t, 1)\tilde{u}(t, 1)+\frac{k}{(s(t)-a)^2}\tilde{u}_y(t, 0)\tilde{u}(t, 0) + \frac{k}{(s(t)-a)^2}\int_0^1|\tilde{u}_y(t)|^2 dy\\
& =\frac{a_0}{s(t)-a}\biggl(\sigma(\tilde{u}_1(t, 1))(\sigma(\tilde{u}_1(t, 1))-\varphi(s(t)))-\sigma(\tilde{u}_2(t, 1))(\sigma(\tilde{u}_2(t, 1))-\varphi(s(t)))\biggr)\tilde{u}(t, 1)\\
& -\frac{1}{s(t)-a}\biggl(\beta(h(t)-H\tilde{u}_1(t, 0))-\beta(h(t)-H\tilde{u}_2(t, 0))\biggr)\tilde{u}(t, 0) + \frac{k}{(s(t)-a)^2}\int_0^1|\tilde{u}_y(t)|^2 dy\\
& \geq  -\frac{a_0}{s(t)-a}\varphi(s(t))|\tilde{u}(t, 1)|^2 - \frac{c_{\beta}H}{s(t)-a}|\tilde{u}(t, 0)|^2 + \frac{k}{(s(t)-a)^2}\int_0^1|\tilde{u}_y(t)|^2 dy. 
\end{align*}
Combining this inequality with (3.11), it follows that  
\begin{align*}
& \frac{1}{2}\frac{d}{dt}|\tilde{u}(t)|^2_{L^2(0, 1)}  + \frac{k}{(s(t)-a)^2}\int_0^1|\tilde{u}_y(t)|^2 dy\\
= &\int_0^1 \frac{ys_t(t)}{s(t)-a}f_{y}(t) \tilde{u}(t)dy +\frac{a_0}{s(t)-a}\varphi(s(t))|\tilde{u}(t, 1)|^2 +  \frac{c_{\beta}H}{s(t)-a}|\tilde{u}(t, 0)|^2.
\tag{3.12}
\end{align*}
Here, we use the Sobolev's embedding therem in one dimensional case: 
\begin{align*}
& |u(y)|^2 \leq C_e|u|_{H^1(0, 1)}|u|_{L^2(0, 1)} \mbox{ for }u\in H^1(0, 1) \mbox{ and }y\in [0, 1],
\tag{3.13}
\end{align*}
where $C_e$ is a positive constant in Sobolev's embedding. By using (3.13), we have 
\begin{align*}
& \frac{1}{2}\frac{d}{dt}|\tilde{u}(t)|^2_{L^2(0, 1)}  + \frac{k}{(s(t)-a)^2}\int_0^1|u_y(t)|^2 dy\\
= & \int_0^1 \frac{ys_t(t)}{s(t)-a}f_{y}(t) \tilde{u}(t)dy + C_e\left(\frac{a_0c_{\varphi}}{s(t)-a} + \frac{c_{\beta}H}{s(t)-a}\right)|\tilde{u}(t)|_{H^1(0, 1)}|\tilde{u}(t)|_{L^2(0, 1)}. 
\tag{3.14}
\end{align*}
Taking $C_2=C_e(a_0c_{\varphi}+c_{\beta}H)$ and using Young's inequality leads to  
\begin{align*}
& \int_0^1 \frac{y s_t(t)}{s(t)-a}f_{y}(t) \tilde{u}(t)dy \\
\leq & |s_t|_{L^{\infty}(0, T)}|\tilde{u}(t)|_{L^2(0, 1)} \left( \int_0^1\frac{1}{(s(t)-a)^2}|f_y(t)|^2dy \right)^{1/2}, \\
& \frac{C_2}{s(t)-a}|\tilde{u}|_{H^1(0, 1)}|\tilde{u}|_{L^2(0, 1)} \leq \frac{C_2}{s(t)-a}(|\tilde{u}_y|_{L^2(0, 1)}|\tilde{u}|_{L^2(0, 1)} + |\tilde{u}|^2_{L^2(0, 1)})\\
\leq & \frac{k}{2(s(t)-a)^2}|\tilde{u}_y|^2_{L^2(0, 1)} + \left(\frac{C^2_2}{2k} +\frac{C_2}{s(t)-a}\right)|\tilde{u}|^2_{L^2(0, 1)}.
\end{align*}
Now, we put $\delta_s$ such that $s(t)-a\geq \delta_s$ for $t\in [0, T]$. 
By (3.14), we obtain 
\begin{align*}
& \frac{1}{2}\frac{d}{dt}|\tilde{u}(t)|^2_{L^2(0, 1)}  + \frac{k}{2(s(t)-a)^2}\int_0^1|\tilde{u}_y(t)|^2 dy\\
\leq & \frac{|f_y(t)|^2_{L^2(0, 1)}}{2}+ 
\left(\frac{|s_t|^2_{L^{\infty}(0, T)}}{2\delta^2_s}+\frac{C^2_2}{2k}+\frac{C_2}{\delta_s}\right)|\tilde{u}(t)|^2_{L^2(0, 1)}.
\tag{3.15}
\end{align*}
Now, by setting $$I(t):=\frac{1}{2}|\tilde{u}(t)|^2_{L^2(0, 1)}+ \frac{k}{2(L-a)^2}\int_0^t|\tilde{u}_y(\tau)|^2_{L^2(0, 1)}d\tau$$ for $t\in [0, T]$, we have 
\begin{align*}
& \frac{d}{dt}I(t) \leq \frac{|f_y(t)|^2_{L^2(0, 1)}}{2} + \left(\frac{|s_t|^2_{L^{\infty}(0, T)}}{2\delta^2_s} + \frac{ C^2_2}{2k}+\frac{C_2}{\delta_s} \right)I(t). 
\tag{3.16}
\end{align*}
Denote by $C_3$ the coefficient of $I(t)$ arising in the right-hand side. Using Gronwall's inequality to (3.16) gives
\begin{align*}
& I(t) \leq \left( \frac{1}{2}\int_0^t|f_y(\tau)|^2_{L^2(0, 1)}d\tau \right) e^{C_3T} \mbox{ for }t\in [0, T].
\end{align*}
This implies that that there exists a small $T_1\leq T$ such that $\Gamma_{T_1}$ is a contraction mapping on $W^{1,2}(Q(T))\cap L^{\infty}(0, T;H^1(0, 1))$. Therefore, by Banach's fixed point theorem we can find $\tilde{u}\in W^{1,2}(Q(T))\cap L^{\infty}(0, T;H^1(0, 1))$ such that 
$\Gamma_{T_1}(\tilde{u})=\tilde{u}$. In other words, we can find a solution $\tilde{u}$ of $(\mbox{AP})_{\tilde{u}_0, s, h}$ on $[0, T_1]$. Since $T_1$ is indepedent of the choice of initial value, by repeating the argument of the local existence result, we can extend the solution $\tilde{u}$ beyond $T_1$. This argument completes the proof of the Lemma.
\end{proof}

As next step,  for given $s\in W^{1,2}(0, T)$ with $a<s<L$ on $[0, T]$, we construct a solution to problem $(\mbox{AP})_{\tilde{u}_0, s, h}$.

\begin{lemma}
\label{lem4}
Let $T>0$ and $L>a$. If (A1)-(A5) hold, then, for given $s\in W^{1,2}(0, T)$ with $a<s<L$ on $[0, T]$, the problem $(\mbox{AP})_{\tilde{u}_0, s, h}$ has a unique solution $\tilde{u}$ on $[0, T]$.
\end{lemma}

\begin{proof}
We choose a sequence $\{s_n\} \subset W^{1, \infty}(0, T)$ and $a<\delta <L$ satisfying $s_n(t)-a \geq \delta$ on $[0, T]$ for each $n\in \mathbb{N}$, 
$s_n \to s$ in $W^{1,2}(0, T)$ as $n \to \infty$. By Lemma \ref{lem3} we can take a sequence $\{\tilde{u}_n\}$ of solutions to $(\mbox{AP})_{\tilde{u}_0, s_n, h}$ on $[0, T]$. Then, we see that $t \to \psi^t(\tilde{u}_n(t))$ is absolutely continuous on $[0, T]$ so that $t \to \frac{k}{(s_n(t)-a)^2}|\tilde{u}_{ny}(t)|^2_{L^2(0, 1)}$ is continuous on $[0, T]$. 
First, we have 
\begin{align*}
& \frac{1}{2}\frac{d}{dt}|\tilde{u}_n(t)|^2_{L^2(0, 1)} -\int_0^1 \frac{k}{(s_n(t)-a)^2}\tilde{u}_{nyy}(t) \tilde{u}_n(t) dy = \int_0^1 \frac{y s_{nt}(t)}{s_n(t)-a} \tilde{u}_{ny}(t) \tilde{u}_n(t) dy
\end{align*}
For the second term in the left hand side, it holds that 
\begin{align*}
& -\int_0^1 \frac{k}{(s_n(t)-a)^2}\tilde{u}_{nyy}(t) \tilde{u}_n(t) dy\\
& =\frac{1}{s_n(t)-a}a_0\sigma(\tilde{u}_n(t, 1))(\sigma(\tilde{u}_n(t, 1))-\varphi(s_n(t)))\tilde{u}_n(t, 1)\\
& -\frac{1}{s_n(t)-a}\beta(h(t)-H\tilde{u}_n(t, 0))\tilde{u}_n(t, 0) + \frac{k}{(s_n(t)-a)^2}\int_0^1|\tilde{u}_{ny}(t)|^2 dy.
\end{align*}
Accordingly, by $a_0(\sigma(\tilde{u}_n(t, 1)))^2\tilde{u}_n(t, 1)\geq 0$ we obtain that 
\begin{align*}
& \frac{1}{2}\frac{d}{dt}|\tilde{u}_n(t)|^2_{L^2(0, 1)} + \frac{k}{(s_n(t)-a)^2}\int_0^1|\tilde{u}_{ny}(t)|^2 dy \\
\leq &  \int_0^1 \frac{y s_{nt}(t)}{s_n(t)-a} \tilde{u}_{ny}(t) \tilde{u}_n(t) dy + \frac{1}{s_n(t)-a}a_0\varphi(s_n(t))\sigma(\tilde{u}_n(t, 1))\tilde{u}_n(t, 1) \\
& + \frac{1}{s_n(t)-a}\beta(h(t)-H\tilde{u}_n(t, 0))\tilde{u}_n(t, 0) \mbox{ for} t\in [0, T].
\tag{3.17}
\end{align*}
Using (3.13) it follows that 
\begin{align*}
& \int_0^1 \frac{y s_{nt}(t)}{s_n(t)-a} \tilde{u}_{ny}(t) \tilde{u}_n(t) dy \leq \frac{k}{4(s_n(t)-a)^2}\int_0^1|\tilde{u}_{ny}(t)|^2dy + \frac{|s_{nt}(t)|^2}{k}\int_0^1|\tilde{u}_n(t)|^2dy,
\end{align*}
and 
\begin{align*}
& \frac{1}{s_n(t)-a}a_0\varphi(s(t))\sigma(\tilde{u}_n(t, 1)\tilde{u}_n(t, 1)\leq \frac{a_0c_{\varphi}}{s_n(t)-a}\biggl(|\tilde{u}_n(t, 1)|^2 + \tilde{u}_n(t, 1)\varphi(a)\biggr)\\
\leq & \frac{a_0c_{\varphi}}{s_n(t)-a}\biggl(\frac{3}{2}|\tilde{u}_n(t, 1)|^2+\frac{\varphi^2(a)}{2}\biggr)\\
\leq & \frac{3a_0c_{\varphi}C_e}{2(s_n(t)-a)}\biggl(|\tilde{u}_{ny}(t)|_{L^2(0, 1)}|\tilde{u}_n(t)|_{L^2(0, 1)} + |\tilde{u}_n(t)|^2_{L^2(0, 1)}\biggr)
+ \frac{a_0c_{\varphi}}{s_n(t)-a}\frac{\varphi^2(a)}{2}\\
\leq & \frac{k}{4(s_n(t)-a)^2}|\tilde{u}_{ny}(t)|^2_{L^2(0, 1)} + \biggl(\frac{(3a_0c_{\varphi}C_e)^2}{4k}+\frac{3a_0c_{\varphi}C_e}{2\delta}\biggr)|\tilde{u}_n(t)|^2_{L^2(0, 1)} + \frac{a_0c_{\varphi}}{\delta }\frac{\varphi^2(a)}{2}, 
\end{align*}
and 
\begin{align*}
& \frac{1}{s_n(t)-a}\beta(h(t)-H\tilde{u}_n(t, 0))\tilde{u}_n(t, 0) \leq \frac{c_{\beta}}{s_n(t)-a}|\tilde{u}_n(t, 0)|\\
\leq & \frac{c_{\beta}C_e}{2(s_n(t)-a)}\biggl(|\tilde{u}_{ny}(t)|_{L^2(0, 1)}|\tilde{u}_n(t)|_{L^2(0, 1)} + |\tilde{u}_n(t)|^2_{L^2(0, 1)}\biggr) + 
\frac{c_{\beta}}{2(s_n(t)-a)}\\
\leq & \frac{k}{4(s_n(t)-a)^2}|\tilde{u}_{ny}(t)|^2_{L^2(0, 1)} + \left(\frac{(c_{\beta}C_e)^2}{4k} + \frac{c_{\beta}C_e}{2\delta}\right)|\tilde{u}_n(t)|^2_{L^2(0, 1)} + \frac{c_{\beta}}{2\delta}.
\end{align*}
As a consequence, we see from the above two estimates and (3.17) that 
\begin{align*}
& \frac{1}{2}\frac{d}{dt}|\tilde{u}_n(t)|^2_{L^2(0, 1)} + \frac{k}{4(s_n(t)-a)^2}\int_0^1|\tilde{u}_{ny}(t)|^2 dy \\
\leq & \left( \frac{|s_{nt}(t)|^2}{k}+ \frac{(3a_0c_{\varphi}C_e)^2}{4k}+\frac{3a_0c_{\beta}C_e}{2\delta} + \frac{(c_{\beta}C_e)^2}{4k} + \frac{c_{\beta}C_e}{2\delta}\right)|\tilde{u}_n(t)|^2_{L^2(0, 1)} \\
& + \frac{a_0c_{\varphi}}{\delta }\frac{\varphi^2(a)}{2} + \frac{c_{\beta}}{2\delta} \mbox{ for }t\in[0, T].
\end{align*}
We denote now the coefficient of $|\tilde{u}_n|^2_{L^2(0, 1)}$ in the above inequality by $F(t)$. Then, $F\in L^1(0, T)$ and Gronwall's inequality yields that 
\begin{align*}
& \frac{1}{2}|\tilde{u}_n(t)|^2_{L^2(0, 1)} + \int_0^t \frac{k}{4(s_n(t)-a)^2}|\tilde{u}_{ny}(t)|^2_{L^2(0, 1)} d\tau \\
& \leq \left(\frac{1}{2}|\tilde{u}(0)|^2_{L^2(0, 1)} + \left(\frac{a_0c_{\varphi}}{\delta }\frac{\varphi^2(a)}{2} + \frac{c_{\beta}}{2\delta}\right)T \right)e^{\int_0^tF(\tau)d\tau} \mbox{ for }t\in [0, T].
\tag{3.18}
\end{align*}
Next, for each $n\in \mathbb{N}$ and $h>0$, we can write
\begin{align*}
& \int_0^1 \tilde{u}_{nt}(t) \frac{\tilde{u}_n(t)-\tilde{u}_n(t-h)}{h}dy - \int_0^1 \frac{k}{(s_n(t)-a)^2} \tilde{u}_{nyy}(t)\frac{\tilde{u}_n(t)-\tilde{u}_n(t-h)}{h}dy\\
&=\int_0^1 \frac{y s_{nt}(t)}{s_n(t)-a}\tilde{u}_{ny}(t) \frac{\tilde{u}_n(t)-\tilde{u}_n(t-h)}{h}dy. 
\tag{3.19}
\end{align*}
For the second term of (3.19), we obtain
\begin{align*}
& - \int_0^1 \frac{k}{(s(t)-a)^2} \tilde{u}_{nyy}(t)\frac{\tilde{u}_n(t)-\tilde{u}_n(t-h)}{h}dy\\
=& -\frac{k\tilde{u}_{ny}(t, 1)}{(s_n(t)-a)^2}\frac{\tilde{u}_n(t, 1)-\tilde{u}_n(t-h, 1)}{h} + \frac{k\tilde{u}_{ny}(t, 0)}{(s_n(t)-a)^2}\frac{\tilde{u}_n(t, 0)-\tilde{u}_n(t-h, 0)}{h}\\
& +\int_0^1 \frac{k\tilde{u}_{ny}(t)}{(s(t)-a)^2}\frac{\tilde{u}_{ny}(t)-\tilde{u}_{ny}(t-h)}{h}dy.
\end{align*}
We name as $I_1$, $I_2$ and $I_3$ the three terms in the last identity. We proceed with estimating them from bellow. For the first term $I_1$, using the same notation $g_1$ and $g_2$ cf. the proof of  Lemma \ref{lem1}, it holds that 
\begin{align*}
I_1 & \geq \frac{1}{h} \frac{1}{s_n(t)-a}\left(\int_0^{\tilde{u}_n(t, 1)}a_0\sigma(\xi)(\sigma(\xi)-\varphi(s_n(t))d\xi - \int_0^{\tilde{u}_n(t-h, 1)}a_0\sigma(\xi)(\sigma(\xi)-\varphi(s_n(t))d\xi \right)\\
=& \frac{g_1(s_n(t), \tilde{u}_n(t, 1))-g_1(s_n(t-h), \tilde{u}_n(t-h, 1))}{h}\\
& + \frac{1}{h}\left(\frac{1}{s_n(t-h)-a}-\frac{1}{s_n(t)-a}\right) \int_0^{\tilde{u}_n(t-h, 1)}a_0\sigma(\xi)(\sigma(\xi)-\varphi(s_n(t-h))d\xi \\
& + \frac{1}{h}\frac{1}{s_n(t)-a}\int_0^{\tilde{u}_n(t-h, 1)}\biggl(a_0\sigma(\xi)(\sigma(\xi)-\varphi(s_n(t-h))-a_0\sigma(\xi)(\sigma(\xi)-\varphi(s_n(t))) \biggr)d\xi.
\end{align*}
Next, for the term $I_2$ we have
\begin{align*}
I_2 & \geq \frac{1}{h}\frac{1}{s_n(t)-a}\left(-\int_0^{\tilde{u}_n(t, 0)}\beta(h(t)-H\xi)d\xi + \int_0^{\tilde{u}_n(t-h, 0)}\beta(h(t)-H\xi)d\xi \right)\\
=& \frac{g_2(s_n(t), h(t), \tilde{u}_n(t, 0))-g_2(s_n(t-h), h(t-h), \tilde{u}_n(t-h, 0))}{h}\\
& + \frac{1}{h}\left(-\frac{1}{s_n(t-h)-a}+\frac{1}{s_n(t)-a}\right)\int_0^{\tilde{u}_n(t-h, 0)}\beta(h(t-h)-H\xi)d\xi\\
& - \frac{1}{h}\frac{1}{s_n(t)-a}\int_0^{\tilde{u}_n(t-h, 0)}\biggl(\beta(h(t-h)-H\xi)-\beta(h(t)-H\xi)\biggr)d\xi
\end{align*}
The term $I_3$ can be dealt with as follows
\begin{align*}
I_3 & \geq \frac{1}{h}\frac{k}{2(s_n(t)-a)^2}\left(\int_0^1|\tilde{u}_{ny}(t)|^2dy-\int_0^1|\tilde{u}_{ny}(t-h)|^2dy\right)\\
& =\frac{1}{h}\left(\frac{k}{2(s_n(t)-a)^2}\int_0^1|\tilde{u}_{ny}(t)|^2dy-\frac{k}{2(s_n(t-h)-a)^2}\int_0^1|\tilde{u}_{ny}(t-h)|^2dy\right)\\
& +\frac{1}{h}\left(\frac{k}{2(s_n(t-h)-a)^2}-\frac{k}{2(s_n(t)-a)^2}\right) \int_0^1|\tilde{u}_{ny}(t-h)|^2dy
\end{align*} 
Combining all these lower bounds and using the fact that $t \to k/(s_n(t)-a)^2|\tilde{u}_{ny}(t)|^2$ is continuous on $[0, T]$, we obtain
\begin{align*}
& \liminf_{h\to 0}(I_1+I_2+I_3)\\
\geq & \frac{d}{dt}\psi^t(\tilde{u}_n(t)) + \frac{s_{nt}(t)}{(s_n(t)-a)^2}\int_0^{\tilde{u}_n(t, 1)}a_0\sigma(\xi)(\sigma(\xi)-\varphi(s_n(t))d\xi\\
& + \frac{1}{s_n(t)-a}\varphi'(s_n(t))s_{nt}(t)\int_0^{\tilde{u}_n(t, 1)}\sigma(\xi)d\xi + \frac{s_{nt}(t)}{(s_n(t)-a)^2}\int_0^{\tilde{u}_n(t, 0)}\beta(h(t)-H\xi)d\xi\\
& - \frac{1}{s_n(t)-a}\int_0^{\tilde{u}_n(t, 0)}\beta'(h(t)-H\xi)h_t(t) d\xi+ \frac{k s_{nt}(t)}{(s_n(t)-a)^3}\int_0^1|\tilde{u}_{ny}(t)|^2dy.
\end{align*}
Applying this result to (3.19) and letting $h\to 0$, we observe 
\begin{align*}
& |\tilde{u}_{nt}(t)|^2_{L^2(0, 1)} + \frac{d}{dt}\psi^t(\tilde{u}_n(t))\\
\leq & \int_0^1\frac{y s_{nt}(t)}{s_n(t)-a} \tilde{u}_{ny}(t)\tilde{u}_{nt}(t)dy + \frac{s_{nt}(t)}{(s_n(t)-a)^2} \left|\int_0^{\tilde{u}_n(t, 1)}a_0\sigma(\xi) (\varphi(s_n(t))-\sigma(\xi))d\xi \right| \\
& + \frac{|\varphi'(s_n(t))||s_{nt}(t)|}{s_n(t)-a}\int_0^{\tilde{u}_n(t, 1)}\sigma(\xi)d\xi + \frac{|s_{nt}(t)|}{(s_n(t)-a)^2}\int_0^{\tilde{u}_n(t, 0)}\beta(h(t)-H\xi)d\xi\\
& + \frac{1}{s_n(t)-a}\left| \int_0^{\tilde{u}_n(t, 0)}\beta'(h(t)-H\xi)h_t(t) d\xi\right| + \frac{k |s_{nt}(t)|}{(s_n(t)-a)^3}\int_0^1|\tilde{u}_{ny}(t)|^2dy. 
\end{align*}
Using Lemma \ref{lem1}, we  estimate now from above each of the terms $J_i$ for $1\leq i\leq 6$ that pinpoint each term from the the right-hand side of the above inequality. By using $\sigma(r)\leq |r|+\varphi(a)$ for $r\in \mathbb{R}$ the following upper bounds hold true:
\begin{align*}
J_1 & \leq \frac{1}{2}|\tilde{u}_{nt}(t)|^2_{L^2(0, 1)} + \frac{1}{2}\frac{|s_{nt}(t)|^2}{(s_n(t)-a)^2}|\tilde{u}_{ny}(t)|^2_{L^2(0, 1)}\\
& \leq \frac{1}{2}|\tilde{u}_{nt}(t)|^2_{L^2(0, 1)} + \frac{|s_{nt}(t)|^2}{k}\left(C_0\psi^t(\tilde{u}_n(t))+C_1\right), \\
J_2 & \leq \frac{a_0 |s_{nt}(t)|\varphi(s_n(t))}{2\delta^2}\left(\frac{|\tilde{u}_n(t, 1)|^2}{2} + \tilde{u}_n(t, 1)\varphi(a)\right), \\
& \leq \frac{a_0 |s_{nt}(t)|\varphi(s_n(t))}{2\delta^2}\left(|\tilde{u}_n(t, 1)|^2 + \frac{\varphi^2(a)}{2}\right), \\
J_3 & \leq \frac{c_{\varphi}}{\delta}|s_{nt}(t)|\left(\frac{|\tilde{u}_n(t, 1)|^2}{2} + \tilde{u}_n(t, 1)\varphi(a) \right)\\
& \leq \frac{c_{\varphi}}{\delta}|s_{nt}(t)|\left(|\tilde{u}_n(t, 1)|^2 + \frac{\varphi^2(a)}{2}\right), \\
J_4 & \leq \frac{|s_{nt}(t)|c_{\beta}}{\delta^2}|\tilde{u}_n(t, 0)| \leq \frac{c_{\beta}}{\delta^2} \left(\frac{|s_{nt}(t)|^2}{2}+\frac{|\tilde{u}_n(t, 0)|^2}{2}\right), \\
J_5 & \leq \frac{c_{\beta}}{\delta}|h_t(t)||\tilde{u}_n(t, 0)|\leq \frac{c_{\beta}}{\delta}\left(\frac{|h_{t}(t)|^2}{2}+\frac{|\tilde{u}_n(t, 0)|^2}{2}\right), \\
J_6 & \leq \frac{k|s_{nt}(t)|}{(s_n(t)-a)^3}\int_0^1|\tilde{u}_{ny}(t)|^2dy \leq \frac{2|s_{nt}(t)|}{\delta}\left(C_0\psi^t(\tilde{u}_n(t))+C_1\right).
\end{align*}
Finally, by combining all these estimates, we obtain that 
\begin{align*}
& \frac{1}{2}|\tilde{u}_{nt}(t)|^2_{L^2(0, 1)} + \frac{d}{dt}\psi^t(\tilde{u}_n(t))\\
\leq & \left (\frac{|s_{nt}(t)|^2}{k} + \frac{2|s_{nt}(t)|}{\delta} \right) (C_0\psi^t(\tilde{u}_n(t)+C_1)+ \frac{a_0 |s_{nt}(t)|c_{\varphi}}{2\delta^2}\left(|\tilde{u}_n(t, 1)|^2 + \frac{\varphi^2(a)}{2}\right)\\
& + \frac{c_{\beta}}{\delta^2} \frac{|s_{nt}(t)|^2}{2}
 + \frac{c_{\varphi}|s_{nt}(t)|}{\delta} \left(|\tilde{u}_n(t, 1)|^2 + \frac{\varphi^2(a)}{2}\right) \\
& + \left( \frac{c_{\varphi}}{\delta} + \frac{c_{\beta}}{\delta^2} \right) \frac{|\tilde{u}_n(t, 0)|^2}{2} + \frac{c_{\beta}}{\delta} \frac{|h_{t}(t)|^2}{2}
\mbox{ for }t\in [0, T].
\end{align*}
Therefore, by setting 
\begin{align*}
l(t)& :=\frac{|s_{nt}(t)|^2}{k} + \frac{2|s_{nt}(t)|}{\delta} + \frac{a_0 |s_{nt}(t)|c_{\varphi}}{2\delta^2} + \frac{c_{\varphi}|s_{nt}(t)|}{\delta} + \frac{1}{2}\left(\frac{c_{\varphi}}{\delta} + \frac{c_{\beta}}{\delta^2} \right) \\
& + \frac{\varphi^2(a)}{2}\left( \frac{a_0 |s_{nt}(t)|c_{\varphi}}{2\delta^2} + \frac{c_{\varphi}|s_{nt}(t)|}{\delta}\right)
\end{align*}
and using Gronwall's lemma, we have that 
\begin{align*}
& \frac{1}{2}\int_0^t|\tilde{u}_{nt}(\tau)|^2_{L^2(0, 1)}d\tau + \psi^t(\tilde{u}_n(t)) \\
\leq & \biggl[\psi^0({\tilde{u}(0)}) + \frac{c_{\beta}}{2\delta^2} \int_0^t |s_{nt}(t)|^2 d\tau  + \frac{c_{\beta}}{2\delta}\int_0^t |h_{t}(\tau)|^2 d\tau\\
& + (C_1+1)\int_0^t l(\tau)d\tau \biggr] e^{C_0\int_0^t l(\tau)d\tau}
\mbox{ for }t\in [0, T].
\tag{3.20}
\end{align*}
Therefore, by $l\in L^2(0, T)$ and combining the latter inequality with (A2) we see that the right hand side of (3.20) is bounded. From this result, we infer that the sequence $\{ \tilde{u}_{n} \}$ is bounded in $W^{1,2}(0, T;L^2(0, 1))$ and the sequence $\{ \psi^{(\cdot)}(\tilde{u}_n(\cdot))\}$ is bounded in $L^{\infty}(0, T)$. Finally, this result in combination with Lemma \ref{lem1}, (3.18) and (3.20) means that the sequence $\{ \tilde{u}_n \}$ is bounded in $W^{1,2}(0, T;L^2(0, 1))\cap L^{\infty}(0, T; H^1(0, 1))$.  Therefore, we can take a sequence $\{ n_k \}\subset \{ n\}$ such that for some $\tilde{u}\in W^{1,2}(0, T; L^2(0, 1))\cap L^{\infty}(0, T; H^1(0, 1))$, $\tilde{u}_n \to \tilde{u}$ weakly in $W^{1,2}(0, T; L^2(0, 1))$, weakly -* in $L^{\infty}(0, T; H^1(0, 1))$ and in $C(\overline{Q(T)})$ as $k \to \infty$. By letting $k\to \infty$, we get that 
$\tilde{u}$ is a solution of $(\mbox{AP})_{\tilde{u}_0, s, h}$ on $[0, T]$.
\end{proof}

\section{Local existence }
\label{LG}
In this section, using the results obtained in Section \ref{AP} we first show that (P)$_{\tilde{u}_0, s_0, h}$ has a solution locally in time. Throughout the rest of this section, we assume (A1)-(A5). 
Let $T>0$ and set $L>s_0$. We define the set 
\begin{align*}
& M(T, s_0, a'):=\{ s\in W^{1,2}(0, T) | a'\leq s <L \mbox{ on } [0, T], s(0)=s_0\}.
\end{align*}
Also, for given $s\in M(T, s_0, a')$, we define the operator $\Phi : M(T, s_0, a') \to V(T):=W^{1,2}(0, T; L^2(0, 1))\cap L^{\infty}(0, T; H^1(0, 1))$ by 
$\Phi (s)=\tilde{u}$, where $\tilde{u}$ is a solution of $(\mbox{AP})_{\tilde{u}_0, s, h}$, and the operator $\Gamma_{T}: M(T, s_0, a') \to W^{1,2}(0, T)$ by 
$\Gamma_{T}(s)=s_0 + \int_0^t a_0(\sigma(\Phi(s)(\tau, 1)-\varphi(s(\tau)))d\tau$ for $t\in [0, T]$. Moreover, for any $K>0$ we put 
\begin{align*}
& M_K(T):=\{s\in M(T, s_0, a') | \ |s|_{W^{1,2}(0, T)}\leq K \}.
\end{align*}

This setting is constructed such that, relying on (3.18) and (3.20) in Lemma \ref{lem4}, the inequality in the next Lemma holds true. 

\begin{lemma}
\label{lem5}
Let $T>0$ and $K>0$.  It holds that 
\begin{align*}
& |\Phi(s)|_{W^{1,2}(0, T;L^2(0, 1))} + |\Phi(s)|_{L^{\infty}(0, T;H^1(0, 1))} \leq C \mbox{ for }s\in M_K(T), 
\end{align*}
where $C=C(T,\tilde{u}_0, K, L, h)$ depending on $T$, $\tilde{u}_0$, $K$, $L$ and $h$.
\end{lemma}

By using Lemma \ref{lem5} we show that for some $T>0$, the mapping $\Gamma_{T}$ is a contraction mapping on the closed set of $M_K(T)$ for any $K>0$. 

\begin{lemma}
\label{lem6}
Let $a<a'\leq s_0$ and $K>0$. There exists a positive constant $T_1\leq T$ such that the mapping $\Gamma_{T_1}: M_K(T_1) \to M_K(T_1)$ is well defined. Furthermore,  
the maping $\Gamma_{T_1}$ is a contraction on the closed set $M_K(T_1)$ in $W^{1,2}(0, T)$.
\end{lemma}

\begin{proof}
For $T>0$ and $L>s_0$, let $s\in M(T, s_0, a')$ and $\tilde{u}=\Phi_{T}(s)$. Then, $\tilde{u}$ is a solution of $(\mbox{AP})_{\tilde{u}_0, s, h}$ so that $\sigma(\Phi(s)(t, 1)) \geq \varphi(a)$ for $t\in [0, T]$, and 
\begin{align*}
\Gamma_{T}(s)(t) &= s_0 + \int_0^t a_0(\sigma(\Phi(s)(\tau, 1))-\varphi(s(\tau)))d\tau \\
&\geq s_0 + a_0(\varphi(a)-c_{\varphi})t \mbox{ for }t\in [0, T].
\tag{4.1}
\end{align*}
Here, by (3.13) and Lemma \ref{lem5}, it follows that 
\begin{align*}
& \int_0^t |\tilde{u}(\tau, 1)|^2d\tau \leq  C_e \int_0^t (|\tilde{u}_y|_{L^2(0, 1)}|\tilde{u}|_{L^2(0, 1)} + |\tilde{u}|^2_{L^2(0, 1)})d\tau \\
\leq & C_e \left(|\tilde{u}|_{L^{\infty}(0, T;L^2(0, 1))} \sqrt{t}\left(\int_0^t |\tilde{u}_y|^2_{L^2(0, 1)}d\tau\right)^{1/2} + t |\tilde{u}|^2_{L^{\infty}(0, T;L^2(0, 1))}\right)\\
\leq & \sqrt{t}C_e (1+\sqrt{T})C^2.
\end{align*}
Then, we have that 
\begin{align*}
\Gamma_T(s) & \leq s_0 + a_0\sqrt{t}\biggl(\int_0^t |\Phi(s)(\tau, 1)|^2 d\tau \biggr)^{\frac{1}{2}}\\
& \leq s_0 +a_0t^{\frac{3}{4}}(C_e (1+\sqrt{T})C^2)^{\frac{1}{2}}. 
\tag{4.2}\\
\end{align*}
Hence, we obtain that 
\begin{align*}
& \int_0^t |\Gamma_T(s)|^2 d\tau 
\leq 2s^2_0t + 2a^2_0tT^\frac{3}{2}\biggl(C_e(1+\sqrt{T})C^2 \biggr)
\tag{4.3}
\end{align*}
and 
\begin{align*}
& \int_0^t |\Gamma'_T(s)|^2d\tau \leq a^2_0\int_0^t|\Phi(s)(\tau, 1))|^2 d\tau \\
\leq & a^2_0\sqrt{t}C_e(1+\sqrt{T})C^2.
\tag{4.4}
\end{align*}
Therefore,  by (4.1)-(4.4) we see that there exists $T_0<T$ such that $\Gamma_{T_0}(s)\in M_K(T_0)$. 
Next, let $\tilde{u}_1$ and $\tilde{u}_2$ for $s_1$ and $s_2\in M_K(T_0)$, respectively, and set $\tilde{u}=\tilde{u}_1-\tilde{u}_2$, $s=s_1-s_2$ and $\delta=a'-a$. 
Then, we have that 
\begin{align*}
& \frac{1}{2}\frac{d}{dt}|\tilde{u}(t)|^2_H - \int_0^1 \left( \frac{k}{(s_1(t)-a)^2}\tilde{u}_{1yy}(t)-\frac{k}{(s_2(t)-a)^2}\tilde{u}_{2yy}(t)\right)\tilde{u}(t)dy \\
&=\int_0^1 \left(\frac{y s_{1t}(t)}{s_1(t)-a} \tilde{u}_{1y}(t)-\frac{ys_{2t}(t)}{s_2(t)-a} \tilde{u}_{2y}(t)\right)\tilde{u}(t)dy.
\tag{4.5}
\end{align*}
Regarding the second term of the left hand side of (4.5), we write
\begin{align*}
& - \int_0^1 \left( \frac{k}{(s_1(t)-a)^2}\tilde{u}_{1yy}(t)-\frac{k}{(s_2(t)-a)^2}\tilde{u}_{2yy}(t)\right)\tilde{u}(t)dy\\
=& \int_0^1 \left(\frac{k}{(s_1(t)-a)^2}\tilde{u}_{1y}(t)-\frac{k}{(s_2(t)-a)^2}\tilde{u}_{2y}(t)\right) \tilde{u}_y(t) dy\\
& - \left (\frac{k}{(s_1(t)-a)^2}\tilde{u}_{1y}(t)(t, 1)-\frac{k}{(s_2(t)-a)^2}\tilde{u}_{2y}(t)(t, 1)\right)\tilde{u}(t, 1)\\
& + \left (\frac{k}{(s_1(t)-a)^2}\tilde{u}_{1y}(t, 0)-\frac{k}{(s_2(t)-a)^2}\tilde{u}_{2y}(t, 0)\right)\tilde{u}(t, 0)\\
=:& I_1+I_2+I_3.
\end{align*}
For the term $I_1$, it holds that  
\begin{align*}
& I_1= \frac{k}{(s_1(t)-a)^2}|\tilde{u}_{y}(t)|^2_{L^2(0, 1)}+ \int_0^1 \left(\frac{k}{(s_1(t)-a)^2}-\frac{k}{(s_2(t)-a)^2}\right)\tilde{u}_{2y}(t)\tilde{u}_y (t)dy\\
\geq & \frac{k}{(s_1(t)-a)^2}|\tilde{u}_{y}(t)|^2_{L^2(0, 1)} - \frac{2Lk|s(t)|}{\delta^3(s_1(t)-a)}|\tilde{u}_{2y}(t)|_{L^2(0, 1)}|\tilde{u}_y(t)|_{L^2(0, 1)} \\
\geq & \left(1-\frac{\eta}{2}\right)\frac{k}{(s_1(t)-a)^2}|\tilde{u}_{y}(t)|^2_{L^2(0, 1)} -\frac{k}{2\eta}\left(\frac{2L^2}{\delta^3}\right)^2|s(t)|^2|\tilde{u}_{2y}|^2_{L^2(0, 1)},
\end{align*}
where $\eta $ is arbitrary positive number. The term $I_2$ is handled  as follows:
\begin{align*}
& -\left (\frac{k}{(s_1(t)-a)^2}\tilde{u}_{1y}(t, 1)-\frac{k}{(s_2(t)-a)^2}\tilde{u}_{2y}(t, 1)\right)\tilde{u}(t, 1)\\
=&\left(\frac{a_0\sigma(\tilde{u}_1(t, 1))}{s_1(t)-a}(\sigma(\tilde{u}_1(t, 1))-\varphi (s_1(t))-\frac{a_0\sigma(\tilde{u}_2(t, 1))}{s_2(t)-a}(\sigma(\tilde{u}_2(t, 1))-\varphi (s_2(t))\right)\tilde{u}(t, 1)\\
=& \frac{a_0}{s_1(t)-a}\biggl(\sigma(\tilde{u}_1(t, 1))(\sigma(\tilde{u}_1(t, 1))-\varphi (s_1(t))-\sigma(\tilde{u}_2(t, 1))(\sigma(\tilde{u}_2(t, 1))-\varphi (s_2(t))\biggr)\tilde{u}(t, 1)\\
& + \left(\frac{1}{s_1(t)-a}-\frac{1}{s_2(t)-a}\right)a_0\sigma(\tilde{u}_2(t, 1))(\sigma(\tilde{u}_2(t, 1))-\varphi (s_2(t)))\tilde{u}(t, 1)\\
=& \frac{a_0}{s_1(t)-a}\biggl(\sigma(\tilde{u}_1(t, 1))-\sigma(\tilde{u}_2(t, 1))\biggr)(\sigma(\tilde{u}_1(t, 1))-\varphi (s_1(t)))\tilde{u}(t, 1)\\
& +\frac{a_0}{s_1(t)-a}\sigma(\tilde{u}_2(t, 1))\biggl(\sigma(\tilde{u}_1(t, 1))-\varphi (s_1(t))-\sigma(\tilde{u}_2(t, 1))+\varphi (s_2(t))\biggr)\tilde{u}(t, 1)\\
& + \left(\frac{1}{s_1(t)-a}-\frac{1}{s_2(t)-a}\right)a_0\sigma(\tilde{u}_2(t, 1))(\sigma(\tilde{u}_2(t, 1))-\varphi (s_2(t)))\tilde{u}(t, 1)\\
&=: I_{21}+I_{22}+I_{23}.
\end{align*}
By using (3.13) and (A4), the following inequalitis hold:  
\begin{align*}
|I_{21}| & \leq \frac{a_0C_e}{s_1(t)-a}|\sigma(\tilde{u}_1(t, 1))-\varphi (s_1(t))||\tilde{u}(t)|_{H^1(0, 1)}|\tilde{u}(t)|_{L^2(0, 1)}\\
|I_{22}| & \leq \frac{a_0}{s_1(t)-a}\sigma(\tilde{u}_2(t, 1))\biggl(|\tilde{u}(t, 1)|^2 + |\varphi (s_1(t))-\varphi (s_2(t))|\tilde{u}(t, 1)|\biggr)\\
& \leq \frac{a_0C_e}{s_1(t)-a}\sigma(\tilde{u}_2(t, 1))|\tilde{u}(t)|_{H^1(0, 1)}|\tilde{u}(t)|_{L^2(0, 1)}\\ 
& + \frac{a^2_0C_e}{2(s_1(t)-a)^2}(\sigma(\tilde{u}_2(t, 1))^2|\tilde{u}(t)|_{H^1(0, 1)}|\tilde{u}(t)|_{L^2(0, 1)} + \frac{c^2_{\varphi}}{2}|s(t)|^2\\
|I_{23}| & = \biggl(\frac{s(t)}{(s_1(t)-a)(s_2(t)-a)}\biggr)a_0\sigma(\tilde{u}_2(t, 1))(\sigma(\tilde{u}_2(t, 1))-\varphi (s_2(t))\tilde{u}(t, 1)\\
& \leq \frac{C_e\biggl(a_0\sigma(\tilde{u}_2(t, 1))(\sigma(\tilde{u}_2(t, 1))-\varphi (s_2(t))\biggr)^2}{2\delta^2(s_1(t)-a)^2}|\tilde{u}(t)|_{H^1(0, 1)}|\tilde{u}(t)|_{L^2(0, 1)} + \frac{1}{2}|s(t)|^2. \\
\end{align*}
Accordingly, by adding the above three estimates, we obtain: 
\begin{align*}
& \sum_{k=1}^3|I_{2k}| \\
\leq & \left(\frac{L_1(t)}{s_1(t)-a} + \frac{L_2(t)}{(s_1(t)-a)^2}\right) |\tilde{u}(t)|_{H^1(0, 1)}|\tilde{u}(t)|_{L^2(0, 1)} + \frac{(c^2_{\varphi}+1)}{2}|s(t)|^2 \mbox{ for }t\in [0, T_0], 
\tag{4.6}
\end{align*}
where $L_1(t)=a_0C_e(|\tilde{u}_1(t, 1)|+\varphi(a)+c_{\varphi})+ a_0C_e(|\tilde{u}_2(t, 1)|+\varphi(a))$ and $L_2(t) = a^2_0C_e(|\tilde{u}_2(t, 1)|^2+\varphi^2(a)) + C_e(a^2_0(|\tilde{u}_1(t, 1)|+\varphi(a))^4)/2\delta^2 $. As for $I_2$, we split the term $I_3$ as follows:  
\begin{align*}
& \left (\frac{k}{(s_1(t)-a)^2}\tilde{u}_{1y}(t, 0)-\frac{k}{(s_2(t)-a)^2}\tilde{u}_{2y}(t, 0)\right)\tilde{u}(t, 0)\\
=&-\left(\frac{1}{s_1(t)-a}\beta(h(t)-H\tilde{u}_1(t, 0))-\frac{1}{s_2(t)-a}\beta(h(t)-H\tilde{u}_2(t, 0))\right)\tilde{u}(t, 0)\\
= & -\frac{1}{s_1(t)-a}\biggl(\beta(h(t)-H\tilde{u}_1(t, 0))-\beta(h(t)-H\tilde{u}_2(t, 0))\biggr)\tilde{u}(t, 0)\\
& -\left(\frac{1}{s_1(t)-a}-\frac{1}{s_2(t)-a}\right)\beta(h(t)-H\tilde{u}_2(t, 0))\tilde{u}(t, 0)\\
 =:&  I_{31}+I_{32}.
\end{align*}
Then, by using (3.13) and (A3), we notice that  
\begin{align*}
& \sum_{k=1}^2|I_{3k}|\\
\leq & \left(\frac{c_{\beta}C_eH}{s_1(t)-a} + \frac{c^2_{\beta}C_e}{2\delta^2(s_1(t)-a)^2}\right)|\tilde{u}(t)|_{H^1(0, 1)}|\tilde{u}(t)|_{L^2(0, 1)} + \frac{1}{2}|s(t)|^2 \mbox{ for }t\in [0, T_0]. 
\tag{4.7}
\end{align*}
What concerns the right-hand side of (4.4), we obtain that  
\begin{align*}
& \int_0^1 \left(\frac{y s_{1t}(t)}{s_1(t)-a} \tilde{u}_{1y}(t)-\frac{y s_{2t}(t)}{s_2(t)-a} \tilde{u}_{2y}(t)\right)\tilde{u}(t)dy\\
=& \int_0^1 \frac{y s_{1t}(t)}{s_1(t)-a} \tilde{u}_y(t)\tilde{u}(t)dy+\int_0^1 \frac{ys_t(t)}{s_1(t)-a} \tilde{u}_{2y}(t)\tilde{u}(t)dy\\
& +\int_0^1 \left(\frac{1}{s_1(t)-a} -\frac{1}{s_2(t)-a} \right)ys_{2t}(t)\tilde{u}_{2y}(t)\tilde{u}(t)dy,
\end{align*}
while the three terms are controlled from above in the following way:
\begin{align*}
I_{41} &\leq \frac{\eta k}{2(s_1(t)-a)^2}|\tilde{u}_y(t)|^2_{L^2(0, 1)} + \frac{1}{2\eta k}|s_{1t}(t)|^2|\tilde{u}(t)|^2_{L^2(0, 1)}, \\
I_{42} &\leq \frac{1}{2\delta}\biggl(|s_t(t)|^2+|\tilde{u}_{2y}(t)|^2_{L^2(0, 1)}|\tilde{u}(t)|^2_{L^2(0, 1)}\biggr), \\
I_{43} &\leq \frac{1}{2\delta^2}\biggl(|s(t)|^2|\tilde{u}_2(t)|^2_{L^2(0, 1)}+|s_{2t}(t)|^2|\tilde{u}(t)|^2_{L^2(0, 1)}\biggr), 
\end{align*}
Then, by (4.6) and (4.7) we have 
\begin{align*}
& \frac{1}{2}\frac{d}{dt}|\tilde{u}(t)|^2_{L^2(0, 1)}+(1-\eta)\frac{k}{(s_1(t)-a)^2}|\tilde{u}_y(t)|^2_H\\
\leq & \left(L_1(t) + c_{\beta}C_eH\right) \frac{1}{s_1(t)-a}|\tilde{u}(t)|_{H^1(0, 1)}|\tilde{u}(t)|_{L^2(0, 1)}\\
& + \biggl(L_2(t) + \frac{c^2_{\beta}C_e}{2\delta^2}\biggr) \frac{1}{(s_1(t)-a)^2}|\tilde{u}(t)|_{H^1(0, 1)}|\tilde{u}(t)|_{L^2(0, 1)}\\
& + \left(\frac{1}{2\eta k}|s_{1t}(t)|^2 +  \frac{1}{2\delta}|\tilde{u}_{2y}(t)|^2_{L^2(0, 1)} + \frac{1}{2\delta^2}|s_{2t}(t)|^2 \right)|\tilde{u}(t)|^2_{L^2(0, 1)}\\
& + \left( \frac{c^2_{\varphi}}{2} + 1+ \frac{1}{2\delta^2}|\tilde{u}_2(t)|^2_{L^2(0, 1)} + \frac{k}{2\eta}\left(\frac{2L^2}{\delta^3}\right)^2|\tilde{u}_{2y}|^2_{L^2(0, 1)} \right)|s(t)|^2 + \frac{1}{2\delta}|s_t(t)|^2. 
\tag{4.8}
\end{align*}
Young's inequality together with (3.13) ensure
\begin{align*}
& \left(L_1(t) + c_{\beta}C_eH \right) \frac{1}{s_1(t)-a}|\tilde{u}(t)|_{H^1(0, 1)}|\tilde{u}(t)|_{L^2(0, 1)}\\
\leq & \left(L_1(t) + c_{\beta}C_eH \right)\frac{1}{s_1(t)-a}\biggl(|\tilde{u}_y(t)|_{L^2(0, 1)}|\tilde{u}(t)|_{L^2(0, 1)} + |\tilde{u}(t)|^2_{L^2(0, 1)}\biggr)\\
\leq & \left(L_1(t) + c_{\beta}C_eH \right)\left(\frac{\eta k}{2(s_1(t)-a)^2}|\tilde{u}_y(t)|^2_{L^2(0, 1)} + (\frac{1}{2\eta k} + \frac{1}{\delta})|\tilde{u}(t)|^2_{L^2(0, 1)} \right)
\end{align*}
and 
\begin{align*}
& \left(L_2(t) + \frac{c^2_{\beta}C_e}{2\delta^2}\right) \frac{1}{(s_1(t)-a)^2}|\tilde{u}(t)|_{H^1(0, 1)}|\tilde{u}(t)|_{L^2(0, 1)} \\
\leq & \left(L_2(t) + \frac{c^2_{\beta}C_e}{2\delta^2}\right) \frac{1}{(s_1(t)-a)^2}(|\tilde{u}_y(t)|_{L^2(0, 1)}|\tilde{u}(t)|_{L^2(0, 1)} + |\tilde{u}(t)|^2_{L^2(0, 1)})\\
\leq &  \left(L_2(t) + \frac{c^2_{\beta}C_e}{2\delta^2}\right) \frac{1}{(s_1(t)-a)^2}\frac{\eta k}{2}|\tilde{u}_y(t)|^2_{L^2(0, 1)} + \frac
{1}{\delta^2}(\frac{1}{2\eta k} + 1)|\tilde{u}(t)|^2_{L^2(0, 1)}, 
\end{align*}
Here, by (3.13) and Lemma \ref{lem5}, we have that 
\begin{align*}
|\tilde{u}_i(t, 1)|^2 & \leq C_e (|\tilde{u}_{iy}(t)|_{L^2(0, 1)}|\tilde{u}_i(t)|_{L^2(0, 1)} + |\tilde{u}_i(t)|^2_{L^2(0, 1)}) \\
& \leq 2C_eC^2 \mbox{ for }t\in [0, T_0], 
\tag{4.9}
\end{align*}
where $C$ is the same constant as in Lemma \ref{lem5}. Then, by (4.9) we notice that $L_1$ and $L_2$ are bounded in $L^{\infty}(0, T)$. 
Accordingly, by applying these results to (4.8) and taking a suitable $\eta=\eta_0$, we have 
\begin{align*}
& \frac{1}{2}\frac{d}{dt}|\tilde{u}(t)|^2_{L^2(0, 1)}+\frac{1}{2}\frac{k}{(s_1(t)-a)^2}|\tilde{u}_y(t)|^2_{L^2(0, 1)}\\
\leq & \left(L_1(t) + c_{\beta}C_eH\right) \left(\frac{1}{2\eta_0 k} + \frac{1}{\delta} \right)|\tilde{u}(t)|^2_{L^2(0, 1)} \\
& + \left(L_2(t) + \frac{c^2_{\beta}C_e}{2\delta^2}\right)\frac{1}{\delta^2}\left(\frac{1}{2\eta_0 k} + 1\right)|\tilde{u}(t)|^2_{L^2(0, 1)}\\
& + \left(\frac{1}{2\eta_0 k}|s_{1t}(t)|^2 +  \frac{1}{2\delta}|\tilde{u}_{2y}(t)|^2_{L^2(0, 1)} + \frac{1}{2\delta^2}|s_{2t}(t)|^2 \right)|\tilde{u}(t)|^2_{L^2(0, 1)}\\
& + \left(\frac{c^2_{\varphi}}{2} +1 + \frac{1}{2\delta^2}|\tilde{u}_2(t)|^2_{L^2(0, 1)} + \frac{k}{2\eta_0}\left(\frac{2L^2}{\delta^3}\right)^2|\tilde{u}_{2y}(t)|^2_{L^2(0, 1)}\right)|s(t)|^2 + \frac{1}{2\delta}|s_t(t)|^2. 
\tag{4.10}
\end{align*}
Now, we put the summation of all coefficient of $|\tilde{u}|^2_{L^2(0, 1)}$ by $L_3(t)$ for $t\in [0, T_0]$, and take 
$L_4(t)=c^2_{\varphi}/2+ 1+ |\tilde{u}_2(t)|^2_{L^2(0, 1)}/2\delta^2 + k(4L^4|\tilde{u}_{2y}(t)|^2_{L^2(0, 1)})/2\eta_0\delta^6+ 1/2\delta$. Then, we have  
\begin{align*}
& \frac{1}{2}\frac{d}{dt}|\tilde{u}(t)|^2_{L^2(0, 1)}+\frac{1}{2}\frac{k}{(s_1(t)-a)^2}\int_0^1|\tilde{u}_y(\tau)|^2_{L^2(0, 1)} d\tau \\
\leq & L_3(t)|\tilde{u}(t)|^2_{L^2(0, 1)} + L_4(t)(|s(t)|^2 + |s_t(t)|^2). 
\tag{4.11} \mbox{ for }t\in [0, T_0].
\end{align*}
Here, using Lemma \ref{lem5}, (4.2) and $s_i\in M_K(T_0)$ for $i=1, 2$ we see that $L_3\in L^1(0, T_0)$ and $L_4\in L^{\infty}(0, T_0)$. 
Therefore, Gronwall's inequality guarantees that 
\begin{align*}
& \frac{1}{2}|\tilde{u}(t)|^2_{L^2(0, 1)}+\frac{1}{2}\frac{k}{(s_1(t)-a)^2}\int_0^t|\tilde{u}_y(\tau)|^2_{L^2(0, 1)} d\tau \\ 
\leq & \left( |L_4|_{L^{\infty}(0, T_0)}|s|^2_{W^{1,2}(0, T)} \right) e^{\int_0^t L_3(\tau)d\tau} \mbox{ for }t\in [0, T_0]. 
\tag{4.12}
\end{align*}  
By using (4.12) we show that there exists $T^*<T_0$ such that $\Gamma_{T^*}$ is a contraction mapping on the closed subset of $M_K(T^*)$. 
To do so, from the subtraction of the time derivatives of $\Gamma_{T_0}(s_1)$ and $\Gamma_{T_0}(s_2)$ and relying on (3.13) and (4.12), we have for $T_1<T_0$ the following estimate: 
\begin{align*}
& |(\Gamma_{T_1}(s_1))_t-(\Gamma_{T_1}(s_2))_t|_{L^2(0, T_1)}\\
=& a_0|\sigma(\tilde{u_1}(\cdot, 1))-\varphi(s_1(\cdot)) - \sigma((\tilde{u_2}(\cdot, 1))-\varphi(s_2(\cdot))|_{L^2(0, T_1)}\\ 
\leq & a_0 \biggl(|\tilde{u}_1(\cdot, 1)-\tilde{u}_2(\cdot, 1)|_{L^2(0, T_1)} + c_{\varphi}|s|_{L^2(0, T_1)} \biggr)\\
\leq & a_0c_{\varphi}T_1|s_t|_{L^2(0, T_1)} + a_0\sqrt{C_e}\biggl(\int_0^{T_1}(|\tilde{u}_y|_{L^2(0, 1)} |\tilde{u}|_{L^2(0, 1)} + |\tilde{u}|^2_{L^2(0, 1)})dt \biggr)^{1/2}\\
\leq & a_0c_{\varphi}T_1|s_t|_{L^2(0, T_1)} \\
 & + C_3\biggl (\varepsilon|s|_{W^{1,2}(0, T_1)}+\frac{1}{\varepsilon}\sqrt{T_1}|s|_{W^{1,2}(0, T_1)} + \sqrt{T_1}|s|_{W^{1,2}(0, T_1)}\biggr),
\tag{4.13}
\end{align*}
where $C_3$ is a positive constant and $\varepsilon$ is an arbitrary positive number. 
We obtain 
\begin{align*}
& |\Gamma_{T_1}(s_1)-\Gamma_{T_1}(s_2)|_{L^2(0, T_1)}\\
\leq & T_1\biggl( a_0c_{\varphi}T_1|s_t|_{L^2(0, T_1)} + C_3\biggl (\varepsilon|s|_{W^{1,2}(0, T_1)}+(\frac{1}{\varepsilon}+1)\sqrt{T_1}|s|_{W^{1,2}(0, T_1)}\biggr)\biggr).
\tag{4.14}
\end{align*}
Therefore, by (4.13) and (4.14) and taking a sufficiently small number $\varepsilon$ we see that there exists $T^*<T_0$ such that $\Gamma_{T^*}$ is a contraction mapping on a closed subset of $M_K(T^*)$. 
\end{proof}

From Lemma \ref{lem6}, by applying Banach's fixed point theorem, there exists $s\in M_K(T^*)$, where $T^*$ is the same as in Lemma \ref{lem6} such that
$\Gamma_{T^*}(s)=s$. This implies that $(\mbox{P})_{\tilde{u}_0, s_0, h}$ has a unique solution $(s, \tilde{u})$ on $[0, T^*]$. Thus, we can prove the existence and uniqueness of a solution of $(\mbox{P})_{\tilde{u}_0, s_0, h}$ locally in time. This shows that  Theorem \ref{t2} holds. Finally, by introducing the variable 
\begin{align*}
& u(t, z)=\tilde{u}\left(t, \frac{z-a}{s(t)-a}\right) \mbox{ for } z\in [a, s(t)].
\tag{4.15}
\end{align*}
we see that a pair of the function $(s, u)$ is a solution of $(\mbox{P})_{u_0, s_0, h}$ on $[0, T^*]$. 
To prove Theorem \ref{t1} completely, we still must ensure  the boundedness of a solution to $(\mbox{P})_{u_0, s_0, h}$. 

\begin{lemma}
\label{lem7}
Let $T>0$, and $(s, u)$ be a solution of $(\mbox{P})_{u_0, s_0, h}$ on $[0, T]$. Then, $\varphi(a)\leq u(t)\leq |h|_{L^{\infty}(0, T)}H^{-1}$ on $[a, s(t)]$ for $t\in [0, T]$.
\end{lemma}

\begin{proof}
First, from (1.1), we have 
\begin{align*}
& \frac{1}{2}\frac{d}{dt}\int_a^{s(t)}|[-u(t)+\varphi(a)]^{+}|^2 dz -\frac{s_t}{2}|[-u(t, s(t))+\varphi(a)]^+|^2 \\
&  + k\int_a^{s(t)} u_{zz}(t)[-u(t)+\varphi(a)]^+ dz =0 \mbox{ for a.e.}t\in [0, T].
\tag{4.15}
\end{align*}
By $a<s$ on $[0, T]$ and $\varphi' \geq 0$ in (A4), we note that $\varphi(s(t))-\varphi(a) >0$ on $[0, T]$. 
Hence, for the second term in the left hand side, if $u(t, s(t))<\varphi(a)$, then $-\sigma(u(t, s(t)))+\varphi(s(t))=-\varphi(a)+\varphi(s(t))>0$ so that 
\begin{align*}
-\frac{s_t}{2}|[-u(t, s(t))+\varphi(a)]^+|^2 & = \frac{a_0}{2}(-\sigma(u(t, s(t)))+\varphi(s(t))|[-u(t, s(t))+\varphi(a)]^+|^2 \geq 0.
\end{align*}
Also, by the boundary conditions (1.2) and (1.3) it follows that 
\begin{align*}
& ku_z(t, s(t))[-u(t, s(t))+\varphi(a)]^+ \\
= & -\sigma(u(t, s(t)))s_t(t)[-u(t, s(t))+\varphi(a)]^+ \\
= & a_0\sigma(u(t, s(t))(-\sigma(u(t, s(t))+\varphi(s(t)))[-u(t, s(t))+\varphi(a)]^+ \\
\end{align*}
and 
\begin{align*}
&-ku_z(t, a)[-u(t, a)+\varphi(s(t))]^+ =\beta(h(t)-Hu(t, a))[-u(t, a)+\varphi(s(t))]^+. 
\end{align*}
Since $\sigma\geq 0$, $\varphi(s(t))-\varphi(a)>0$ and $\beta \geq 0$ we note that both expressions are positive. Therefore, we obtain that 
\begin{align*}
& \frac{d}{dt}\int_a^{s(t)}|[-u(t)+\varphi(a)]^{+}|^2 dz + k\int_a^{s(t)}|[-u(t)+\varphi(a)]_z^+|^2 dz \leq 0 \mbox{ for a.e. }t\in [0, T].
\tag{4.16}
\end{align*}
Integrating (4.16) over $[0, T]$, we see that $|[-u(t)+\varphi(a)]^+|^2_{L^2(a, s(t))}=0$ for $t\in [0, T]$
which implies $u(t)\geq \varphi(a)$ on $[a, s(t)]$ for $t\in [0, T]$.
Next, we show that $u(t) \leq |h|_{L^{\infty}(0, T)}H^{-1}$ on $[a, s(t)]$ for $t\in [0, T]$. From (1.1), we first obtain 
\begin{align*}
& \frac{1}{2}\frac{d}{dt}|u(t)|^2_{L^2(a, s(t))} +\frac{1}{2}s_t(t)|u(t, s(t))|^2\\
& + k\int_a^{s(t)}|u_z(t)|^2dz -\beta(h(t)-Hu(t, a))u(t, a)=0 \mbox{ for a.e. }t\in [0, T].
\tag{4.17}
\end{align*}
Here, by $u(t, s(t))=\frac{s_t(t)}{a_0}+\varphi(s(t))$ and $u(t, s(t))\geq \varphi(a)$ on $[0, T]$ it holds that 
\begin{align*}
\frac{s_t(t)}{2}|u(t, s(t))|^2 & = \frac{1}{2}\left(\frac{|s_t(t)|^2}{a_0} + \varphi(s(t))s_t(t)\right)u(t, s(t))\\
& \geq \frac{\varphi(a)}{2a_0}|s_t(t)|^2-\frac{c_{\varphi}}{2}|s_t(t)|u(t, s(t))\\
& \geq \frac{\varphi(a)}{4a_0}|s_t(t)|^2-\frac{a_0c^2_{\varphi}}{4\varphi(a)}u^2(t, s(t))
\end{align*}
and 
\begin{align*}
 -\beta(h(t)-Hu(t, a))u(t, a)&=\beta(h(t)-Hu(t, a)\frac{h(t)-Hu(t, a)}{H} -\beta(h(t)-Hu(t, a)\frac{h(t)}{H}\\
& \geq -\beta(h(t)-Hu(t, a)\frac{h(t)}{H}.
\end{align*}
Hence, the above two results and (4.17) leads to 
\begin{align*}
& \frac{1}{2}\frac{d}{dt}|u(t)|^2_{L^2(a, s(t))} +\frac{\varphi(a)}{4a_0}|s_t(t)|^2 + k\int_a^{s(t)}|u_z(t)|^2dz \\
& \leq \frac{a_0c^2_{\varphi}}{4\varphi(a)}u^2(t, s(t))+ \beta(h(t)-Hu(t, a))\frac{h(t)}{H} \mbox{ for a.e. }t\in [0, T].
\tag{4.18}
\end{align*}
By Sobolev's embedding theorem in one dimension, it follows that 
\begin{align*}
&\frac{a_0c^2_{\varphi}}{4\varphi(a)}u^2(t, s(t)) \leq \frac{a_0c^2_{\varphi}}{4\varphi(a)} C'_e|u(t)|_{H^1(a, s(t))}|u(t)|_{L^2(a, s(t))}\\
& \leq \frac{a_0c^2_{\varphi}C'_e}{4\varphi(a)}(|u_z(t)|_{L^2(a, s(t))}|u(t)|_{L^2(a, s(t))}+|u(t)|^2_{L^2(a, s(t))})\\
& \leq \frac{k}{2}|u_z(t)|^2_{L^2(a, s(t))} + \left( \frac{1}{2k}\left(\frac{a_0c^2_{\varphi}C'_e}{4\varphi(a)}\right)^2+\frac{a_0c^2_{\varphi}C'_e}{4\varphi(a)}\right)|u(t)|^2_{L^2(a, s(t))},
\tag{4.19}
\end{align*}
where $C'_e$ is a positive constant in Sobolev's embedding theorem. Therefore, by (4.19), (4.18) becomes 
\begin{align*}
& \frac{1}{2}\frac{d}{dt}|u(t)|^2_{L^2(a, s(t))} +\frac{\varphi(a)}{4a_0}|s_t(t)|^2 + \frac{k}{2}\int_a^{s(t)}|u_z(t)|^2dz \\
& \leq \left( \frac{1}{2k} \left(\frac{a_0c^2_{\varphi}C'_e}{4\varphi(a)}\right)^2+ \frac{a_0c^2_{\varphi}C'_e}{4\varphi(a)}\right)|u(t)|^2_{L^2(a, s(t))} + c_{\beta}\frac{|h|_{L^{\infty}(0, T)}}{H}.
\tag{4.20}
\end{align*}
Integrating (4.20) over $[0, T]$ we see that $s_t\in L^2(0, T)$. 
Now, using a similar argument as in the proof for the lower bound and $\sigma(u(t, s(t))=u(t, s(t))$ we have that 
\begin{align*}
& \frac{1}{2}\frac{d}{dt}\int_a^{s(t)}|[u(t)-|h|_{L^{\infty}(0, T)}H^{-1}]^{+}|^2 dz -\frac{s_t}{2}|[u(t, s(t))-|h|_{L^{\infty}(0, T)}H^{-1}]^+|^2 \\
&  -k\int_a^{s(t)} u_{zz}(t)[u(t)-|h|_{L^{\infty}(0, T)}H^{-1}]^+ dz =0 \mbox{ for a.e. }t\in[0, T].
\tag{4.21}
\end{align*}
By noting from $\sup_{r\in \mathbb{R}}\varphi(r)\leq |h|_{L^{\infty}(0, T)}H^{-1}$ in (A4) that 
\begin{align*}
& -ku_z(t, s(t))[u(t, s(t))-|h|_{L^{\infty}(0, T)}H^{-1}]^+ \\ 
= & u(t, s(t))s_t[u(t, s(t))-|h|_{L^{\infty}(0, T)}H^{-1}]^+ \\
=& a_0u(t, s(t))(u(t, s(t))-\varphi(s(t)))[u(t, s(t))-|h|_{L^{\infty}(0, T)}H^{-1}]^+ \\
\geq & a_0|h|_{L^{\infty}(0, T)}H^{-1}(|h|_{L^{\infty}(0, T)}H^{-1}-\sup_{r\in \mathbb{R}}\varphi(r))[u(t, s(t))-|h|_{L^{\infty}(0, T)}H^{-1}]^+ \geq 0, 
\end{align*}
and 
\begin{align*}
& ku_z(t, a)[u(t, a)-|h|_{L^{\infty}(0, T)}H^{-1}]^+ \\ 
= & -\beta(h(t)-Hu(t, a))[u(t, a)-|h|_{L^{\infty}(0, T)}H^{-1}]^+ =0, 
\end{align*}
we can write (4.21) as follows:
\begin{align*}
& \frac{1}{2}\frac{d}{dt}\int_a^{s(t)}|[u(t)-|h|_{L^{\infty}(0, T)}H^{-1}]^{+}|^2 dz + k\int_a^{s(t)} |[u(t)-|h|_{L^{\infty}(0, T)}H^{-1}]^+_z|^2 dz\\
& \leq \frac{s_t(t)}{2}|[u(t, s(t))-|h|_{L^{\infty}(0, T)}H^{-1}]^+|^2 \mbox{ for a.e. }t\in[0, T].
\tag{4.22}
\end{align*}
Similarly to (4.19), we obtain 
\begin{align*}
& \frac{s_t(t)}{2}|[u(t, s(t))-|h|_{L^{\infty}(0, T)}H^{-1}]^+|^2 \\
\leq & \frac{s_t(t)C'_e}{2}(|U_z(t)|_{L^2(a, s(t))}|U(t)|_{L^2(a, s(t))}+|U(t)|^2_{L^2(a, s(t))})\\
\leq & \frac{k}{2}|U_z(t)|^2_{L^2(a, s(t))} + \left(\frac{1}{2k} \left(\frac{s_t(t)C'_e}{2}\right)^2 + \frac{s_t(t)C'_e}{2}\right)|U(t)|^2_{L^2(a, s(t))},
\end{align*}
where $U(t, z)=[u(t, z)-|h|_{L^{\infty}(0, T)}H^{-1}]^+$ for $(t, z)\in Q_s(T)$. 
We put the coefficient of $|u(t)|^2_{L^2(a, s(t))}$ by $G(t)$. Then, $s_t\in L^2(0, T)$ so that we see that $G\in L^1(0, T)$. Finally, by applying the above to (4.22) and using Gronwall's inequality we get 
\begin{align*}
& \frac{1}{2}|[u(t)-|h|_{L^{\infty}(0, T)}H^{-1}]^{+}|^2_{L^2(a, s(t))}+ \frac{k}{2}\int_0^t|[u(t)-|h|_{L^{\infty}(0, T)}H^{-1}]^+_z|^2_{L^2(a, s(t))}dt\\
& \leq \left(\frac{1}{2}|[u_0-|h|_{L^{\infty}(0, T)}H^{-1}]^{+}|^2_{L^2(a, s_0)} \right) e^{\int_0^tG(\tau)d\tau}=0 \mbox{ for } t\in [0, T].
\end{align*}
This means that $u(t) \leq |h|_{L^{\infty}(0, T)}H^{-1}$ on $[a, s(t)]$ for $t\in [0, T]$.  Thus, Theorem \ref{t1} is finally proven. 
\end{proof}

\section*{Acknowledgments} The authors thank T. Aiki (Tokyo) for fruitful discussions. KK is supported by Grant-in-Aid No.16K17636, JSPS.


\begin{thebibliography}{99}
\bibitem{AMSS}
T.~Aiki, Y.~Murase, N.~Sato, K.~Shirakawa, A mathematical model for a
hysteresis appearing in adsorption phenomena, S\=urikaisekikenky\=usho
K\=oky\=uroku, vol. 1856 (2013), 1--12.

\bibitem{AI-Mur}
T.~Aiki, Y.~Murase, On a large time behavior of a solution to a one-dimensional
free boundary problem for adsorption phenomena, J. Math. Anal. Appl.,
vol. 445 (2017), 837--854.

\bibitem{AM}
T.~Aiki, A.~Muntean. 
Existence and uniqueness of solutions to a mathematical model predicting sevice life of concrete structures, 
Adv. Math. Sci. Appl., vol. 19 (2009), 109--129. 

\bibitem{AM1}
T.~Aiki, A.~Muntean.
Large time behavior of solutions to concrete carbonation problem, 
Communications on Pure and Applied Analysis, vol. 9 (2010), 1117--1129

\bibitem{AM2}
T.~Aiki, A.~Muntean.
A free-boundary problem for concrete carbonation : Rigorous justification of $\sqrt{t}$--law of propagation, 
Interfaces and Free Boundaries, vol. 15 (2013), 167--180. 



\bibitem{FMP}
A. Fasano, G. Meyer, M. Primicerio, On a problem in the polymer industry: theoretical and numerical investigation of swelling, SIAM J. Appl. Math. vol. 17 (1986), 945--960.

\bibitem{FM}
A.~Fasano, A.~Mikelic, 
The 3D flow of a liquid through a porous medium with adsorbing and swelling granules, 
Interfaces and Free Boundaries, vol. 4 (2002), 239--261. 

\bibitem{FMA}
T.~Fatima, A.~Muntean, T.~Aiki. Distributed space scales in a semilinear
reaction-diffusion system including a parabolic variational inequality: A
well-posedness study, Adv. Math. Sci. Appl., vol. 22 (2012), 295--318.

\bibitem{FT}
A.~Friedman, A.~Tzavaras, A quasilinear parabolic system arising in modelling
of catalytic reactors. J. Differential Equations, vol. 70 (1987), 167--196.

\bibitem{Kenmochi}
N.~Kenmochi, Solvability of nonlinear evolution equations with time-dependent
  constraints and applications, Bull. Fac. Education, Chiba Univ., vol. 30 (1981), 1--87.

\bibitem{KASM}
K.~Kumazaki, T.~Aiki, N.~Sato, Y.~Murase, 
Multiscale model for moisture transport with adsorption phenomenon in concrete materials, 
Applicable Analysis, vol. 97 (2017), 41--54. 

\bibitem{MB}
A.~Muntean, M.~B{\"o}hm.
A moving boundary problem for concrete carbonation: global existence and uniqueness of solutions.
J. Math. Anal. Appl., vol. 350 (2009), no.1, 234--251.

\bibitem{MN}
A.~Muntean, M.~Neuss-Radu, A multiscale {G}alerkin approach for a class of
  nonlinear coupled reaction-diffusion systems in complex media, J. Math. Anal.
  Appl., vol. 37 (2010), 705--718.

\bibitem{SAMS}
N.~Sato, T.~Aiki, Y.~Murase, K.~Shirakawa, A one dimensional free boundary
  problem for adsorption phenomena, Netw. Heterog. Media, vol. 9 (2014), 655--668.

\bibitem{Evaporation}
B.~W.~van de Fliert, R.~van der Hout, 
A generalized Stefan problem in a diffusion model with evaporation, 
SIAM J. Appl. Math., vol. 60 (2000), no. 4, 1128--1136. 

\bibitem{disso-precipi}
T.~L.~van Noorden, I.~ S.~Pop, 
A Stefan problem modelling crystal dissolution and precipitation, 
IMA J. Appl. Math., 2007, 1--19. 

\bibitem{Helmig}
T. L. van Noorden  I. S. Pop  A. Ebigbo  R. Helmig, An upscaled model for biofilm growth in a thin strip, Water Resources Reserach, vol. 46 (2010), 1--14.

\bibitem{Setzer}
M. J. Setzer, Micro-ice-lens formation in porous solid,
Journal of Colloid and Interface Science,
vol. 243 (2001), no.  1, 193-201,

\bibitem{Xie}
X. Weiqing, The {S}tefan problem with a kinetic condition at the free boundary, 
SIAM J. Math. Anal. vol. 21 (1990), no. 2, 362--373. 

\bibitem{Z}
M.~Zaal, Cell swelling by osmosis: a variational approach,  
Interfaces and Free Boundaries, vol. 14 (2012), 487--520. 

\end{thebibliography}
\end{document}